\documentclass[11pt,a4paper]{article}
\usepackage[utf8]{inputenc}
\usepackage{a4wide} 
\usepackage[utf8]{inputenc}
\usepackage{subfig}
\usepackage{bbm}
\usepackage[tbtags]{amsmath}
\usepackage{graphicx}
\usepackage{pgfplots}
\usepackage{subfloat}
\usepackage{epstopdf}
\usepackage{algcompatible}
\usepackage{adjustbox}
\usepackage{autobreak}
\usepackage{tikz}
\usetikzlibrary{shapes,positioning,backgrounds}
\usepackage{titlesec}
\usepackage{csquotes}
\usepackage{adjustbox}
\usepackage{extarrows}
\allowdisplaybreaks
\RequirePackage{ifthen}
\RequirePackage[T1]{fontenc}
\RequirePackage[ngerman,english]{babel}
\RequirePackage{amsmath}
\RequirePackage{amssymb}
\RequirePackage{amsthm}
\RequirePackage{amsfonts}
\RequirePackage{graphicx}
\RequirePackage{color}
\RequirePackage{listings} 
\usepackage{epstopdf}

\usepackage{algcompatible}
\RequirePackage[section,Algorithm]{algorithm}
\RequirePackage{emptypage}

\usepackage[backend=biber,sorting=nyt]{biblatex}
\renewbibmacro{in:}{}
\DeclareFieldFormat[article]{title}{\textit{#1}} 
\DeclareFieldFormat[article]{journaltitle}{#1}
\theoremstyle{plain}						
\newtheorem{theorem}{Theorem}[section]
\newtheorem{lemma}[theorem]{Lemma}

\numberwithin{equation}{section}

\titleformat{\section}[block]{\normalfont\bfseries}{\thesection.}{0.5em}{}
\titlespacing{\section}{0pc}{1pc}{1pc}

\titleformat{\subsection}[block]{\normalfont\bfseries}{\thesubsection.}{0.5em}{}
\titlespacing{\subsection}{0pc}{1pc}{1pc}

\begin{document}
\title{Control Strategies for Transport Networks under Demand Uncertainty}
\author{Simone G\"ottlich\footnotemark[1], \; Thomas Schillinger\footnotemark[1]}

\footnotetext[1]{University of Mannheim, Department of Mathematics, 68131 Mannheim, Germany (goettlich@uni-mannheim.de, schillinger@uni-mannheim.de)}

\date{ \today }

\maketitle

\begin{abstract}
\noindent
In this article, we consider transport networks with uncertain demands. 
Network dynamics are given by linear hyperbolic partial differential equations and suitable coupling conditions, while demands are incorporated as solutions to stochastic differential equations. 
For the demand satisfaction, we solve a constrained optimal control problem. Controls in terms of network inputs are then calculated explicitly for different assumptions. Numerical simulations are performed to underline the theoretical results.
\end{abstract}




\section{Introduction}
Transport networks play an important role for the description of flows between entities, see~\cite{Ahuja1993,Bressan2014} for an overview. Typical examples include road networks, pipeline networks, power grids or production lines. Such networks ensure that there is light in a room after pressing the light switch or that the production line works efficiently and without larger idle times in an automobile plant. Especially, if uncertainty comes into play, there is a strong interest to analyze the perturbed systems in order to understand how the stochasticity influences the dynamics.

There are different approaches such as classical network flows~\cite{Ahuja1993,Ford1962} or dynamic transport networks~\cite{Bressan2014,Apice,Garavello2016,Koch2015} to describe the deterministic system dynamics. The approaches mainly differ in static respectively dynamic considerations of flows and associated effects. In the case of supply networks (e.g. energy or production) hyperbolic partial differential equations (PDEs) have been established \cite{Banda2006,Gas1,Gottlich2005,Strom1,Leugering2002} using appropriate coupling conditions at nodes. For production lines there also exist network approaches on different scales, i.e. from a microscopic perspective using queuing theory \cite{ProductionQueue} or discrete event systems \cite{DiscreteEvent} up to a macroscopic perspective governed by PDEs of hyperbolic type, as shown in~\cite{ADMProduction}.

In this work, we focus on macroscopic models using densities as the quantity of interest and model the dynamics on each arc in the network by linear hyperbolic partial differential equations. The optimal control problem under consideration is to find the optimal input into the system such that stochastic demands are satisfied. The latter are represented by solutions to stochastic differential equations (SDEs). In contrast to~\cite{Lux2021,Lux1}, where the setting of an Ornstein-Uhlenbeck process has been considered, we focus on a different stochastic process, i.e. the Jacobi process, to model the demand. Originally, the Jacobi process was used to determine interest rates on financial markets \cite{Delbaen}. Recently, the Jacobi process has been also applied in the case of electricity markets to either model electricity prices~\cite{Filipovic} or to investigate intraday electricity demand~\cite{Korn}.

For the identification of the optimal control we pursue an explicit representation under suitable assumptions on the network dynamics. In the special case of the linear advection equation on a single line only, similar investigations have been presented in \cite{Lux1}. In this work, we allow for a wider class of linear hyperbolic PDEs where additional complexity for the computation of the control arises due to the network structure. Depending on different levels of information we distinguish between a non-updated scenario (MS1), a scenario where we update the control with a given frequency (MS2) and a scenario in which we additionally update the conditions at the nodes (MS3). All scenarios are analyzed in detail and also studied from a numerical point of view.

The paper is organized as follows: Section $\ref{descr}$ introduces the optimal control model. Afterwards, we investigate a Jacobi process with time-dependent mean reversion level in Section $\ref{cha2}$. A deeper insight into the network dynamics and the explicit computation of the optimal inflow is presented in Section $\ref{cha4}$. Section $\ref{cha.3}$ deals with the discussion of the different levels of information. Finally, in Section $\ref{cha.num}$ we present a numerical study for the network model and also give a comparison to the Ornstein-Uhlenbeck demand process.

\section{Model description}
\label{descr}
The considered model consists of three parts as can be seen in Figure $\ref{SchematicD}$.  
First, the quantity of interest is the time-dependent control variable $u(t)$ describing the inflow into the network. The control variable will be chosen such that stochastic demands are satisfied. Second, we face a transport network consisting of a single source and a finite number of internal nodes $J$ and end nodes $C$. We assume that networks are connected, directed and tree-structured graphs. The dynamics for the densities $z^{(i)}(x,t)$ on network arc $i$ are governed by a linear hyperbolic partial differential equation equipped with initial conditions $z^{(i)}_0(x)$. The functions $f^{(i)}$ and $g^{(i)}$ denote the flux function and a damping function for arc $i$, respectively. In Section $\ref{cha4}$, we discuss the different choices. 
Due to the network structure we denote by $v_i$ the subsequent node of arc $i$.
For a fixed node $v_i$ we have to impose distribution parameters $\alpha_{i,k}(t)$ describing the share of the flux directed from arc $i$ to arc $k$ at time $t$. We require that the $\alpha_{i,k}(t)$ sum up to 1 for any fixed node $v_i$ to ensure flux conservation. The choice of the distribution parameters will be extensively discussed in Section  $\ref{cha.3}$.
Third, at the sinks of the transport network we assume uncertain demands modeled via stochastic processes $(D^{(v_i)}_t)_{t\in[t_0,T]}$. In Section $\ref{cha2}$, we theoretically investigate the Jacobi process. A numerical comparison to an Ornstein-Uhlenbeck process in Section $\ref{cha.num}$ highlights the key differences. 

\vspace{0.2cm}
\tikzset{
  circ/.style={
    circle,
    fill=white!10,
    draw=black
  }
}

\begin{figure}[htb]
\fcolorbox{black}{white}{
\begin{minipage}[t]{0.23\textwidth}
\centering Source node $v_0$\\ 
\vspace{0.5cm}
\Huge \resizebox{1.5cm}{1.4cm}{\textcolor{teal}{\textbf{?}}}
\vspace{0.5cm}
\normalsize
\centering \\Inflow control $u(t)$
\end{minipage}
\begin{minipage}[t]{0.025\textwidth}
\vspace{1.3cm}
{\Large $\longrightarrow$}
\end{minipage}
\begin{minipage}[t]{0.33\textwidth}
\centering Transport network\\
\resizebox{4.1cm}{2.4cm}{
\begin{tikzpicture}[thick]
\node[circ] (Q1){$v_1$};
\node[circ,left=0.6 of Q1] (Q0){$v_0$};
\node[circ, above right =1 of Q1] (Q2){$v_2$};
\node[circ, right =0.7 of Q2] (Q4){$v_4$};
\node[circ, below = 1.4 of Q2] (Q3) {$v_3$};
\node[circ, right = 0.7 of Q3] (Q5) {$v_5$};
\node[circ, above right = 0.9 of Q3] (Q6) {$v_6$};

\node[draw=none,fill=none, above right = 0.45 of Q4] (Q7) {};
\node[draw=none,fill=none, left = 5.3 of Q7] (Q8) {};
\node[draw=none,fill=none, below = 3.2 of Q8](Q9) {};
\node[draw=none,fill=none, below =3.2 of Q7](Q10){};

\path[->]
(Q0) edge (Q1)
(Q1) edge (Q2)
(Q2) edge (Q4)
(Q1) edge (Q3)
(Q3) edge (Q5)
(Q3) edge (Q6);
\path[]
(Q7) edge (Q8)
(Q8) edge (Q9)
(Q9) edge (Q10)
(Q10) edge (Q7);
\end{tikzpicture}
}
\centering Hyperbolic PDEs for densities $z^{(i)}(x,t)$ 
\end{minipage}
\begin{minipage}[t]{0.025\textwidth}
\vspace{0.6cm}
{\Large $\longrightarrow$}\\
\vspace{-0.3cm}
{\Large $\longrightarrow$}\\
\vspace{-0.25cm}
{\Large $\longrightarrow$}
\end{minipage}
\begin{minipage}[t]{0.28\textwidth}
\centering Demand uncertainty\\
    \centering
    \includegraphics[width=3.3cm]{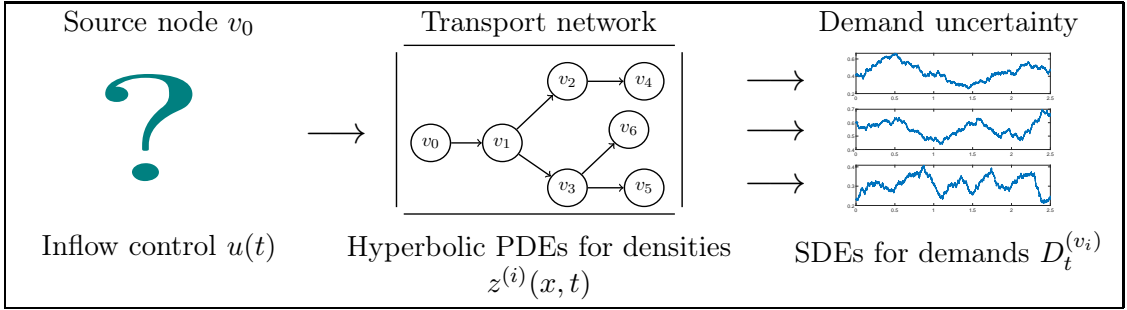}\\
\centering SDEs for demands $D^{(v_i)}_t$
\end{minipage}}
\caption{Schematic presentation of the model components.}
\label{SchematicD}
\end{figure}

To control the inflow into the network $f^{(1)}(z^{(1)}(0,t),t)$, we aim to minimize the expected quadratic deviation of the demand from the actual supply $f^{(i)}(z^{(i)}(1,t),t)$. This leads to a demand tracking type cost function and hence to the following stochastic optimal control problem: 
Consider a finite time horizon $T$ and a network where all arcs have length $1$. The demand processes $(D^{(v_i)}_t)_{t \in [t_0,T]}, ~ v_i\in C$ are defined on a probability space $(\Omega, \mathcal{A}, P)$ which are equipped with a family of filtrations $(\mathcal{F}_t)_{t \in [t_0,T]}$. Let $\hat{t}\leq t_0$ be the time from which the latest demand information is available. Then, the problems reads for all $x\in (0,1)$ and $t \in [t_0,T]$: 
\begin{subequations}\label{ControlProblem11}
\begin{align}
    \underset{u \in L^2}{\min} & \sum_{\lbrace i: v_i \in C \rbrace} \int_{t_0}^T \mathbb{E}\left[\left(D_s^{(v_i)} - f^{(i)}(z^{(i)}(1,s),s)\right)^2 ~\Big|~ \mathcal{F}_{\hat{t}}\right] ds  \label{ObjectiveFunc}\\
\text{s.t.}~~ & z_t^{(i)}(x,t) + f^{(i)}(z^{(i)}(x,t),t)_x + g^{(i)}(z^{(i)}(x,t),t) = 0, ~~~~ \forall i \text{ s.t. } v_i\in J \cup C \label{NetworkDynamics}\\
        & z^{(i)}(x,t_0) = z_0^{(i)}(x), ~~~~ \forall i \text{ s.t. } v_i \in J \cup C \label{initialData}\\
        & f^{(1)}(z^{(1)}(0,t),t) = u(t) \label{boundaryData}\\
        &f^{(k)}(z^{(k)}(0,t),t) =  \alpha_{i,k}(t)f^{(i)}(z^{(i)}(1,t),t), ~~~~ \forall i \text{ s.t. } v_i \in J,~ k\text{ outgoing arc of }v_i \label{fluxCons1}\\
        & \sum_{k \text{ outgoing arc of } v_i} \alpha_{i,k}(t) = 1,~~~ \forall i \text{ s.t. } v_i \in J\label{fluxCons2}\\
        & dD_t^{(v_i)} = \kappa^{(v_i)} \left(\theta^{(v_i)}(t) - D_t^{(v_i)}\right) + \sigma^{(v_i)} \sqrt{D_t^{(v_i)}\left(1-D_t^{(v_i)}\right)} dW_t^{(v_i)},~D_0^{(v_i)} = d_0^{(v_i)}, v_i \in C \label{Demand}
\end{align}
\end{subequations}
In the following section, we address the demand modeling via the Jacobi process in (\ref{Demand}). Following the ideas presented in~\cite{Lux1}, we intend to reformulate the cost function and then explicitly compute the control $u(t)$ depending on the dynamics of the hyperbolic PDE.

\section{Demand modeling using a Jacobi process}
\label{cha2}
Due to the inherent stochastic nature of demand, there exist various approaches to capture the behavior. However, in the course of the day or a year, an underlying pattern can be observed. For example, considering gas or electricity consumption, the demand is larger in the mornings or afternoons than at night. 
In most applications only positive demands occur. This applies to production systems but also to gas, water and electricity networks. It also seems reasonable to assume that there exist a maximum demand. 
Therefore, we require a demand process that follows a mean level and only takes values in a bounded interval. 

A stochastic process that captures most of these characteristics is the Jacobi process given by the solution to the following stochastic differential equation (SDE)
\begin{align}
\label{Jacobi1}
    dZ_t &= \kappa (\theta - Z_t)dt + \sigma\sqrt{Z_t(1-Z_t)}dW_t , ~~~~ Z_{t_0} = z_0,
\end{align}
where $\kappa,~ \sigma>0$ and $\theta, ~z_0 \in \mathbb{R}$ are parameters and $(W_t)_{t \in [t_0,T]}$ is a Brownian motion. The equation consists of a deterministic drift term $\kappa(\theta - Z_t)dt$ which pushes the process back to the mean reversion level $\theta$, where $\kappa$ determines how fast the mean reversion occurs. The second part is the stochastic diffusion $\sigma\sqrt{Z_t(1-Z_t)}dW_t$ which is mainly influenced by the Brownian motion $W_t$. The scale of this stochastic influence is governed by $\sigma$ and by the distance of the Jacobi process to the border values, which are 0 and 1 in this case. The closer the process approaches one of the borders, the smaller gets the stochastic influence of the diffusion part. 
The Jacobi process admits for the Markov property and belongs to the wide class of Pearson diffusion processes consisting of a deterministic drift term and a stochastic diffusion term. Other members of this class, as for instance the Ornstein-Uhlenbeck process or the CIR-Process, have been used to model uncertain demands in various applications \cite{Dohi, Lux1, Lux2, Ouaret}. The Ornstein-Uhlenbeck process is given by the solution to the SDE
\begin{align}
\label{OU-Process}
    d\hat{Z}_t = \hat{\kappa}(\hat{\theta}(t) - \hat{Z}_t)dt + \hat{\sigma}dW_t,~~~ \hat{Z}_{t_0} = \hat{z}_0,
\end{align}
where $(W_t)_{t \in [t_0,T]}$ is a Brownian motion, $\hat{z}_0$ an initial demand and $\hat{\kappa}, \hat{\sigma}$ are positive constants. The function $\hat{\theta}$ models the time-dependent mean reversion level of the Ornstein-Uhlenbeck process. For application purposes they appear very popular since they are easy to handle, admit for an explicit solution and allow for the most basic properties. However, the Ornstein-Uhlenbeck process as well as the CIR-Process do not allow for bounds or only for bounds from below. Therefore, and supported by the recent work \cite{Korn} where a parameter fitting for a Jacobi process in an electricity market has been provided, we concentrate on the Jacobi process and come back to the Ornstein-Uhlenbeck process in Section $\ref{cha.num}$.

The Jacobi process given by the solution of $(\ref{Jacobi1})$ is bounded and stays in the interval $[0,1]$. By a linear transformation the Jacobi process can be shifted to any bounded interval $[\tilde{\alpha},\tilde{\beta}]$:
\begin{align*}
    \tilde{Z}_t = \tilde{\alpha} + (\tilde{\beta}-\tilde{\alpha})Z_t .
\end{align*}
Therefore, and to keep notation simple, in the following we restrict the Jacobi process to $[0,1]$. 

For the Jacobi process in $(\ref{Jacobi1})$ the transition probability for moving from state $x$ at time $s$ to $y$ at time $t$ is given by 
\begin{align*}
  p(x,s;y,t) &= \sum_{n=0}^\infty k_n \psi_n(x) \psi_n(y)w(y)e^{-\eta_n(t-s)},  
\end{align*}
where
\begin{eqnarray*}
 &   k_n = \frac{(a+b+2n-1)\Gamma(a+n)\Gamma(a+b+n-1)}{n!\Gamma(a)^2\Gamma(b+n)},\\ [1ex]
    & w(x) = x^{a-1}(1-x)^{b-1}, \quad
    \psi_n(x) = \sum_{k=0}^n (-1)^k \binom{n}{k}\frac{\Gamma(a+b+n-1+k)\Gamma(a)}{\Gamma(a+b+n-1)\Gamma(a+k)}x^k 
    \end{eqnarray*}
    with $\Gamma(\cdot)$ being the Gamma function and
    \begin{align*}
    a = \frac{2\kappa \theta}{\sigma^2}>0, \quad
 b = \frac{2\kappa(1-\theta)}{\sigma^2},\quad
 \eta_n = \kappa n + \frac{\sigma^2}{2}n(n-1).
        \end{align*}
Using the transition probabilities the conditional expectation for the Jacobi process and $t_0<t$ can be calculated by
\begin{align}
\label{condfirstMoment}
    \mathbb{E}[Z_t | Z_{t_0} = z_0] = \theta + (z_0 - \theta) e^{-\kappa(t-t_0)}.
\end{align}
The conditional second moment is given by
\begin{align}
\begin{split}
    \label{condSecMoment}
    \mathbb{E}[Z_t^2 | Z_{t_0} = z_0] =& \frac{(2 \kappa \theta + \sigma^2)\theta}{2\kappa + \sigma^2}  + \frac{2\kappa \theta + \sigma^2}{\kappa + \sigma^2}(z_0-\theta)e^{-\kappa(t-t_0)} \\
    &~~~ + \left(z_0^2 - \frac{2\kappa \theta + \sigma^2}{\kappa + \sigma^2}z_0 + \frac{\kappa \theta (2\kappa + \sigma^2)}{(2\kappa + \sigma^2)(\kappa + \sigma^2)}  \right)e^{-(2\kappa + \sigma^2)(t-t_0)}.
\end{split}
\end{align}
Both derivations are well known and can be found in e.g. \cite{Delbaen}. For a deeper discussion of the Jacobi process we refer to \cite{JacobiG}. The calculation of the first two moments of the Jacobi process will enable us at the end of this section to find a deterministic representation of the a priori stochastic optimal control problem (\ref{ControlProblem11}).

\subsection{A Jacobi process with time-dependent mean reversion level}
So far, the mean reversion level has been assumed to be constant. However, in many applications the averaged demand is not constant in time. Considering electricity demands or the demand of manufactured goods, there are predictable fluctuations within a day or also within a year. So we extend $(\ref{Jacobi1})$ by adding a time dependency into the mean reversion level
\begin{align}
\label{Jacobi2}
    dZ_t &= \kappa (\theta(t) - Z_t)dt + \sigma\sqrt{Z_t(1-Z_t)}dW_t, ~~~~ Z_{t_0}=z_0.
\end{align}
The parameter $\theta(t)$ is now considered as a function in time. Existence and uniqueness of solutions to $(\ref{Jacobi2})$ can be guaranteed under similar assumptions as for $(\ref{Jacobi1})$ finding appropriate estimates for $\theta$, see \cite{Sorensen}.

Unfortunately, there is no explicit distribution of solution to neither $(\ref{Jacobi1})$ nor $(\ref{Jacobi2})$. However, it is possible to calculate a limit distribution for $t \rightarrow \infty$ in which the solution to $\eqref{Jacobi1}$ is distributed as a Beta distribution with parameters $\alpha = \frac{2\kappa \theta}{\sigma^2}$ and $\beta = \frac{2\kappa (1-\theta)}{\sigma^2}$, see \cite{JacobiG}.

For an illustration of the Jacobi processes, we plot the evolution with and without time-dependent mean reversion levels and different parameters in Figure $\ref{ExamplesJacobi}$. The simulation is performed using a truncated Euler-Maruyama scheme and will be further explained in Section $\ref{cha.num}$.

For a constant mean-reversion level the processes fluctuate around $\theta$ where the oscillations increase for larger $\sigma$ and smaller $\kappa$. It can be observed that especially if $\sigma$ is large or if $\kappa$ is small, the Jacobi process approaches its boundaries. In Figure $\ref{ExamplesJacobi}$ (c) and (d) the processes follow the mean-reversion level with a small delay, except for the representation where $\sigma =1.8$ is chosen. The stochastic influence dominates the mean reverting behaviour and therefore shows larger deviations from the mean-reversion level.

To proceed with the transport network analysis, we are interested in finding the first two conditional moments for $(\ref{Jacobi2})$. Therefore, we note that the conditional expectation of the Jacobi process can be rewritten by
\begin{align*}
    \mathbb{E}\left[Z_t | Z_{t_0} = z_0\right] &= \mathbb{E}\left[Z_{t_0} + \int_{t_0}^t dZ_s | Z_{t_0} = z_0\right] \\
    &= z_0 +  \mathbb{E}\left[\int_{t_0}^t \kappa(\theta(s) - Z_s)ds | Z_{t_0} = z_0\right] + \mathbb{E}\left[\int_{t_0}^t \sqrt{Z_s(1-Z_s)}dW_s | Z_{t_0} = z_0\right].
\end{align*}
Since
\begin{align*}
    \mathbb{E}\left[\int_{t_0}^t \left(\sqrt{Z_s(1-Z_s)} \right)^2ds | Z_{t_0} = z_0 \right] \leq \frac{t-t_0}{4}<\infty
\end{align*}

\definecolor{mycolor1}{rgb}{0.00000,0.44700,0.74100}%
\definecolor{mycolor2}{rgb}{0.85000,0.32500,0.09800}%
\definecolor{mycolor3}{rgb}{0.92900,0.69400,0.12500}%
\definecolor{mycolor4}{rgb}{0.49400,0.18400,0.55600}%

	\begin{minipage}{0.5\textwidth}
		\centering
		\scalebox{.52}{
}%
		\captionof*{figure}{\small{(a) $\kappa=2$, $\theta(t)=0.4$}}
	\end{minipage}
\begin{minipage}{0.5\textwidth}
	\centering
	\scalebox{.52}{
		%
}%
	\captionof*{figure}{\small{(b) $\sigma=0.6$, $\theta(t)=0.4$}}
	
\end{minipage}\\
\vspace{0.4cm}

\begin{minipage}{0.5\textwidth}
	\centering
	\scalebox{.52}{
		%
%
	}
	\captionof*{figure}{\small{(c) $\kappa=2$, $\theta(t)=0.4+0.25\sin(t)$}}
\end{minipage}
	\begin{minipage}{0.5\textwidth}
		\centering
		\scalebox{.52}{
			%
%
		}
		\captionof*{figure}{\small{(d) $\sigma=0.6$, $\theta(t)=0.4+0.25\sin(t)$}}
	\end{minipage}
			\captionof{figure}{Influence of different parameters on the Jacobi process with constant mean reversion level $\theta(t)=0.4$ and time-dependent mean reversion level $\theta(t)=0.4+0.25\sin(t)$.}
			\label{ExamplesJacobi}
			\vspace{0.6cm}
			
the martingale property of the Ito integral (see e.g. \cite{SDEIto:}), yields
\begin{align*}
    \mathbb{E}\left[\int_{t_0}^t \sqrt{Z_s(1-Z_s)}dW_s | Z_{t_0} = z_0\right] = 0.
\end{align*}
Therefore, it holds
\begin{align*}
    \mathbb{E}\left[Z_t | Z_{t_0} = z_0\right] = \mathbb{E}\left[\breve{Z}_t | \breve{Z}_{t_0} = z_0 \right] = \breve{Z}_t
\end{align*}
for the solution to the deterministic linear inhomogeneous initial value problem
\begin{align}
\label{lininhomODE}
    \frac{d}{dt} \breve{Z}_t = \kappa (\theta(t) - \breve{Z}_t), ~~~ \breve{Z}_{t_0} = z_0.
\end{align}
The solution to $(\ref{lininhomODE})$ is explicitly known and given by
\begin{align*}
    \breve{Z}_t = z_0 e^{-\kappa(t-t_0)} + \kappa \int_{t_0}^t e^{-\kappa(t-s)}\theta(s)ds.
\end{align*}
Summarizing, we obtain for the conditional expectation of the Jacobi process with time varying mean reversion level
\begin{align}
\label{condfirstMomentTimevar}
    \mathbb{E}[Z_t | Z_{t_0} = z_0] = z_0 e^{-\kappa(t-t_0)} + \kappa \int_{t_0}^t e^{-\kappa(s-t_0)}\theta(s)ds.
\end{align}
For the reformulation of the optimal control problem, we also require an explicit representation of the second conditional moment of the Jacobi process with time varying mean reversion level. It is given by
\begin{align}
\begin{split}
        \mathbb{E}[Z_t^2 | Z_{t_0}=z_0]  &=  \int_{t_0}^t (2\kappa \theta(s) + \sigma^2)\left(z_0e^{-\kappa(s-t_0)} + \kappa \int_{t_0}^s  \theta(r)e^{-\kappa(s-r)}dr \right)e^{-(2\kappa + \sigma^2)(t-s)}ds \\
    &~~~~+ z_0^2e^{-(2\kappa + \sigma^2)(t-t_0)}.  \label{condsecMomentTimevar}
\end{split}
\end{align}
The computational details are provided in the Appendix. In the following, the demands $D_t^{(v_i)}$ will be described by individual Jacobi processes with time-dependent mean reversion level.

\subsection{Deterministic reformulation of the optimal control problem}
After having investigated the Jacobi process properly, we can decompose the expectation in the objective function of (\ref{ControlProblem11}) in the following way:
\begin{align*}
    &\mathbb{E}\left[\left(D_t^{(v_i)} - f^{(i)}(z^{(i)}(1,t),t) \right)^2  | D_{t_0}^{(v_i)} = d_0^{(v_i)}\right] \\
    &= \mathbb{E}\left[\left(D_t^{(v_i)}\right)^2  | D_{t_0}^{(v_i)} = d_0^{(v_i)}\right] - 2 \mathbb{E}\left[D_t^{(v_i)} | D_{t_0}^{(v_i)} = d_0^{(v_i)}\right] f^{(i)}(z^{(i)}(1,t),t) + \left(f^{(i)}(z^{(i)}(1,t),t) \right)^2 .
\end{align*}
Further applying the results for the Jacobi process from $(\ref{condfirstMomentTimevar})$ and $(\ref{condsecMomentTimevar})$ we obtain for the objective function as
\begin{align}
\begin{split}
\label{detproblem}
    &\mathbb{E}\left[\left(D_t^{(v_i)} - f^{(i)}(z^{(i)}(1,t),t) \right)^2  | D_{t_0}^{(v_i)} = d_0^{(v_i)}\right] \\
    =& \int_{t_0}^t \left(2\kappa^{(v_i)} \theta^{(v_i)}(s) + \left(\sigma^{(v_i)}\right)^2\right)\bigg(d^{(v_i)}_0e^{-\kappa^{(v_i)}(s-t_0)}  \kappa^{v(i)} \int_{t_0}^s  \theta^{(v_i)}(r)e^{-\kappa^{(v_i)}(s-r)}dr \bigg)\\
    &~\cdot e^{-(2\kappa^{(v_i)} + \left(\sigma^{(v_i)}\right)^2)(t-s)}- 2f^{(i)}(z^{(i)}(1,t),t) \kappa^{(v_i)} e^{-\kappa^{(v_i)}(s-t_0)} \theta^{(v_i)}(s) ds\\
    &~+ \left(d_0^{(v_i)}\right)^2e^{-(2\kappa^{(v_i)} + \left(\sigma^{(v_i)}\right)^2)(t-t_0)} - 2f^{(i)}(z^{(i)}(1,t),t)d_0^{(v_i)} e^{-\kappa^{(v_i)}(t-t_0)} + \left(f^{(i)}(z^{(i)}(1,t),t)\right)^2.
\end{split}
\end{align}
This contains exclusively deterministic and known variables and enables us to write the optimal control problem (\ref{ControlProblem11}) without the stochastic demand constraint by using the deterministic explicit representations of the first two conditional moments of the Jacobi demand processes.
Consequently, the minimization problem 
\begin{subequations} \label{ControlProblem111}
\begin{align}
   \underset{u \in L^2}{\min} & \sum_{\lbrace i: v_i \in C \rbrace} \int_{t_0}^T \mathbb{E}\left[\left(D_s^{(v_i)} - f^{(i)}(z^{(i)}(1,s),s)\right)^2 ~\Big|~ \mathcal{F}_{\hat{t}}\right] ds  \label{ofunc}\\
\text{s.t.}~~ & z_t^{(i)}(x,t) + f^{(i)}(z^{(i)}(x,t),t)_x + g^{(i)}(z^{(i)}(x,t),t) = 0, ~~~~ \forall i \text{ s.t. } v_i\in J \cup C \\
        & z^{(i)}(x,t_0) = z_0^{(i)}(x), ~~~~ \forall i \text{ s.t. } v_i \in J \cup C\\
        & f^{(1)}(z^{(1)}(0,t),t) = u(t) \\
        &f^{(k)}(z^{(k)}(0,t),t) =  \alpha_{i,k}(t)f^{(i)}(z^{(i)}(1,t),t), ~~~~ \forall i \text{ s.t. } v_i \in J,~ k\text{ outgoing arc of }v_i \\
        & \sum_{k \text{ outgoing arc of } v_i} \alpha_{i,k}(t) = 1,~~~ \forall i \text{ s.t. } v_i \in J.\\
        &D_0^{(v_i)} = d_0^{(v_i)},~ v_i \in C
\end{align}
\end{subequations}
can be solved deterministically by interpreting (\ref{ofunc}) as given in equation (\ref{detproblem}). 
The following theorem states a general result on minimizing the mean-square error.
\begin{theorem}
\label{meansquareTheorem}
Let $(\Omega, \mathcal{F}, P)$ be a complete probability space and let $\mathcal{G} \subset \mathcal{F}$ a sub-$\sigma$-algebra of $\mathcal{F}$. Let $X,Y$ be two real-valued and square integrable random variables on $\Omega$ where $Y$ is $\mathcal{G}$-measurable. Then $\mathbb{E}[X|\mathcal{G}]$ is the minimizer of the mean-square distance from $X$
\begin{align*}
    \underset{Y}{\min} \left(\mathbb{E}[(X-Y)^2]  \right)
\end{align*}
for all such random variables $Y$.
\end{theorem}
For a proof we refer to \cite{Klenke}. In the setting of $(\ref{ControlProblem111})$ this means that in terms of $L^2$-minimization the outflow flux should be chosen to be the conditional expectation of the demand with respect to the available information, i.e. $\mathbb{E}[D_t^{(v_i)} | D_{\tilde{t}}^{(v_i)} = d_0^{(v_i)}]$ for some $\tilde{t} \in [t_0,T]$. In Section $\ref{cha.3}$ we will discuss the aspect of information availability more carefully while in the next section we consider the different networks dynamics given by linear hyperbolic PDEs.

\section{Network dynamics and optimal input}
\label{cha4}
In this section we focus on the dynamics in the network on all arcs $i$ given by constraint (\ref{NetworkDynamics}) of the optimization problem, i.e. the shape of the functions $f^{(i)}$ and $g^{(i)}$ governed by
\begin{align}
\label{PDEwD}
    z_t^{(i)}(x,t) + f^{(i)}(z^{(i)}(x,t),t)_x + g^{(i)}(z^{(i)}(x,t),t) = 0,
\end{align}
where $f^{(i)}$ denotes the flux function and $g^{(i)}$ is the diffusion function which leads to damping of the transported quantity in the supply system.
We restrict ourselves to linear hyperbolic partial differential equations. For simplicity, all investigations are executed in the 1-1 and 1-2 case and can be generalized to arbitrary tree networks in a straightforward way. Under some assumptions, we will be able to calculate the optimal input explicitly.

\subsection{Linear transport with time-dependent velocity function}
\label{LtTDV}
We start with the consideration of a linear transport dynamic using a time-dependent velocity function, i.e. we choose the flux function $f^{(i)}(z,t) = \lambda_i(t) z$, where $\lambda_i(\cdot)$ is a strictly positive and bounded function. In this section we assume no damping, i.e. $g^{(i)}(z,t)=0$. 
Then, the method of characteristics \cite{LeVeque} leads to a trajectory plot as illustrated in Figure $\ref{traj_timdep}$. Since the trajectories do not intersect, the dynamics are still linear. 

\begin{figure}[h!]
    \centering
    \includegraphics[width=6.3cm]{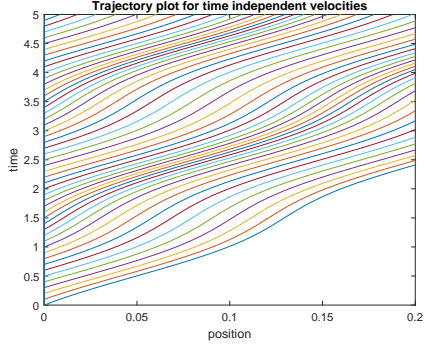}
    \caption{Trajectory plot for a time-dependent velocity function $\lambda(t) = 8 + 3\sin(\pi t)$.}
    \label{traj_timdep}
\end{figure}

\subsubsection{1-1 Network}
In the simple case of a 1-1-network we now demonstrate how the optimal input at the initial node can be determined. 
This procedure can be seen as a straightforward extension of the results in \cite{Lux1} where only one arc and only a constant velocity has been considered.
The structure of the 1-1 network is shown in  Figure $\ref{network1-1}$. 

\tikzset{
  pp/.style={
    rectangle,
    fill=orange!10,
    draw=orange
  }
}
\tikzset{
  ds/.style={
    ellipse,
    fill=blue!10,
    draw=blue
  }
}
\tikzset{
  cs/.style={
    rectangle,
    fill=green!10,
    draw=green
  }
}

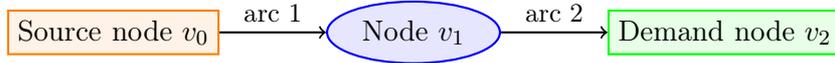
\begin{figure}[H]
\label{12network}
\begin{center}
\begin{tikzpicture}[thick]
\node[pp] (K1) {Source node $v_0$};
\node[ds,right = 1.4cm of K1] (K2) {Node $v_1$};
\node[cs,right = 1.4cm of K2] (K3) {Demand node $v_2$};

\path [->]
(K1) edge node[above]{\small{arc 1}} (K2)
(K2) edge node[above]{\small{arc 2}}(K3);

\end{tikzpicture}
\end{center}
\caption{The supply system as a 1-1-network with one source and one demand node.}
    \label{network1-1}
\end{figure}

To determine the optimal input at some time $t_{in} \in [t_0,T]$ we have to calculate when the corresponding units injected in $t_{in}$ reach node $v_1$ (denoted by $t_1$) and the time $t_2$ when they reach the demand node $v_2$. Having this information obtained, one can calculate the conditional expectation of the demand process at $t_2$. Since the conditional expectation of demand according to Theorem $\ref{meansquareTheorem}$ minimizes the expected quadratic deviation, it is a main driver for the optimal injection into the supply network. In the special case of arcs with length 1, both times $t_1$ and $t_2$ are implicitly given by the system of equations
\begin{align}
\begin{split}
\label{eq1}
     \int_{t_{in}}^{t_1} \lambda_1(r)dr = 1, \quad
    \int_{t_1}^{t_2} \lambda_2(r)dr = 1
\end{split}
\end{align}
where $\lambda_1$ and $\lambda_2$ denote the velocity functions on arc 1 and arc 2, respectively.
These equations can be generalized for arbitrary lengths of the arcs by replacing the right hand sides of $(\ref{eq1})$ by the particular lengths. We present a result that enables us to express $t_1$ and $t_2$ explicitly under certain assumptions.
\begin{lemma}
\label{Lem41}
If the velocity functions $\lambda_i$ have an antiderivative $\Lambda_i, i =1,2$ which is invertible, then for an injection time $t_{in} \in [t_0,T]$, $t_1$ and the output time $t_2$ are given by
\begin{align}
\begin{split}
\label{intermediateTimes2}
    t_1 &= \Lambda_1^{-1}(1+\Lambda_1(t_{in}))\\
    t_2 &= \Lambda_2^{-1}(1+\Lambda_2(\Lambda_1^{-1}(1+\Lambda_1(t_{in})))).
\end{split}
\end{align}
\end{lemma}
\begin{proof}
Using the fundamental theorem of calculus and the second expression of $(\ref{eq1})$ leads to
\begin{align*}
    \Lambda_2(t_2) = \Lambda_2(t_1) + \Lambda_1(t_1) - \Lambda_1(t_{in}).
\end{align*}
The first equation in $(\ref{eq1})$ can be reformulated to express $t_1$ explicitly by
\begin{align*}
    t_1 = \Lambda_1^{-1}(1+\Lambda_1(t_{in})).
\end{align*}
Then, we have
\begin{align*}
    t_2 &= \Lambda_2^{-1}\left(\Lambda_2(t_1) + \Lambda_1(t_1)- \Lambda_1(t_{in})   \right)\\
    &= \Lambda_2^{-1}\left(\Lambda_2(t_1) + 1  \right)\\
    &= \Lambda_2^{-1}\left(1+ \Lambda_2(\Lambda_1^{-1}(1+\Lambda_1(t_{in}))) \right). \qedhere
\end{align*}
\end{proof}
Another aspect of the 1-1 network that should be understood well is the coupling condition of the two arcs in node $v_1$. In our model we require flux conservation, i.e. $f^{(1)}(z^{(1)}(1,t),t) = f^{(2)}(z^{(2)}(0,t),t)$. Since $f^{(1)}$ and $f^{(2)}$ do not necessarily have to coincide, there can be discontinuities in the densities in node $v_1$. These discontinuities have to be taken into account for the calculation of the optimal input. An outflow out of the system of $f^{(2)}(z^{(2)}(1,t_2),t_2)$ is equivalent to a density of $z^{(2)}(1,t_2) = \frac{f^{(2)}(z^{(2)}(1,t_2),t_2)}{\lambda_2(t_2)}$ at the end of arc 2. Since our dynamics are linear, it holds $z^{(2)}(1,t_2) = z^{(2)}(0,t_1)$. But as the velocity function is allowed to be time-dependent, the inflow in arc 2 at time $t_1$, given by $f^{(2)}(z^{(2)}(0,t_1),t_1) = z^{(2)}(0,t_1)\lambda_2(t_1)$, might differ from the system outflow at $t_2$. Therefore, even though the dynamics are linear, the fluxes along the characteristics can vary, while the densities stay constant. The relation between inflow in arc 2 at $t_1$ and outflow out of arc 2 at $t_2$ is given by
\begin{align*}
    \frac{f^{(2)}(z^{(2)}(0,t_1),t_1) }{\lambda_2(t_1)} = z^{(2)}(0,t_1) = z^{(2)}(1,t_2) = \frac{f^{(2)}(z^{(2)}(1,t_2),t_2)}{\lambda_2(t_2)}.
\end{align*}
Flux conservation yields that
\begin{align*}
    z^{(1)}(1,t_1) = \frac{f^{(1)}(z^{(1)}(1,t_1),t_1)}{\lambda_1(t_1)} = \frac{f^{(2)}(z^{(2)}(0,t_1),t_1)}{\lambda_1(t_1)} = f^{(2)}(z^{(2)}(1,t_2),t_2) \frac{\lambda_2(t_1)}{\lambda_1(t_1)\lambda_2(t_2)}.
\end{align*}
Additionally, using the linearity of the dynamics on arc 1 we get
\begin{align*}
    f^{(2)}(z^{(2)}(1,t_2),t_2) 
    &= z^{(1)}(1,t_1) \frac{\lambda_1(t_1)\lambda_2(t_2)}{\lambda_2(t_1)} \\
    &= z^{(1)}(0,t_{in})\frac{\lambda_1(t_1)\lambda_2(t_2)}{\lambda_2(t_1)} \\
    &= f^{(1)}(z^{(1)}(0,t_{in}),t_{in}) \frac{\lambda_1(t_1)\lambda_2(t_2)}{\lambda_1(t_{in})\lambda_2(t_1)} .
\end{align*}

Plugging in that the system outflow should be chosen according to the conditional expectation of the demand process, we obtain the following relation for the optimal system inflow
\begin{align}
\label{optinflow}
    u(t_{in}) =  \frac{\lambda_1(t_{in})}{\lambda_1(t_{1})} \frac{\lambda_2(t_1)}{\lambda_2(t_2)} \mathbb{E}\left[D_{t_2}^{(v_2)} ~ | ~ \mathcal{F}_{\hat{t}}  \right],
\end{align}
where we condition on some time $\hat{t} \in [t_0,t_{in}]$. 
The result can inductively be adapted to 1-1 networks of larger size by multiplying additional factors $\frac{\lambda_i(t_{i-1})}{\lambda_i(t_{i})}$ on the right hand side of $(\ref{optinflow})$ for any additional node $v_i$.

Note that we assume that our system has unlimited capacity, meaning that there is no upper bound on the densities on the network arcs. This assumption avoids that our system reaches a congested state as it has been investigated in \cite{Festa}. 

\subsubsection{1-2 Network}
\label{1-2noDamp}
Since we concentrate on tree networks, we explain next the procedure for the 1-2 case. This network consists of four nodes and three arcs and is presented in Figure $\ref{network1-2}$.
\begin{figure}[h]
\label{12network1}
\begin{center}
\begin{tikzpicture}[thick]
\node[pp] (K1) {Source node $v_0$};
\node[ds,below = 0.8cm of K1] (K2) {Node $v_1$};
\node[cs,below left = 0.8cm of K2] (K3) {Demand node $v_2$};
\node[cs, below right = 0.8cm of K2] (K4) {Demand node $v_3$};

\path [->]
(K1) edge node[right]{\small{arc 1}} (K2)
(K2) edge node[above left]{\small{arc 2}}(K3)
(K2) edge node[above right]{\small{arc 3}}(K4);
\end{tikzpicture}
\end{center}
\caption{The supply system as a 1-2-network with one source and two demand nodes.}
    \label{network1-2}
\end{figure}
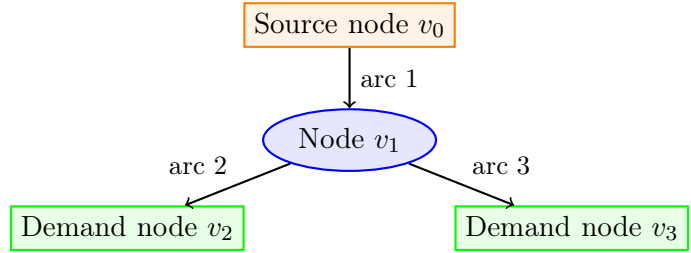

Similar to the 1-1 case, under the assumptions of Lemma $\ref{Lem41}$, i.e. the existence and invertibility of the antiderivatives $\Lambda_i$ of $\lambda_i,~i=1,2,3$, we are able to calculate the times $t_i, ~i=1,2,3$ which are the times when a unit injected at $v_0$ at $t_{in}$ reaches node $v_i$ as 
\begin{align}
\label{intermediateTimes}
\begin{split}
    t_1 &= \Lambda_1^{-1}\left(1+\Lambda_1(t_{in})\right)\\
    t_2 &= \Lambda_2^{-1}\left(1+\Lambda_2(\Lambda_1^{-1}(1+\Lambda_1(t_{in})))\right)\\
    t_3 &= \Lambda_3^{-1}\left(1+\Lambda_3(\Lambda_1^{-1}(1+\Lambda_1(t_{in})))\right).
\end{split}
\end{align}
To calculate the optimal inflow for one of the demand nodes, it is sufficient to consider a subnetwork that contains only the relevant nodes and arcs for the particular demand node. Due to the tree network structure, these subnetworks are 1-1 networks. Hence, following the discussion for the 1-1 network, the optimal input $u_i(t_{in})$ for a single demand node $v_i$ is given by 
\begin{align*}
    u_2(t_{in}) &=  \frac{\lambda_1(t_{in})}{\lambda_1(t_{1})} \frac{\lambda_2(t_1)}{\lambda_2(t_2)} \mathbb{E}\left[D_{t_2}^{(v_2)} ~ | ~ \mathcal{F}_{\hat{t}}  \right]\\
    u_3(t_{in}) &=  \frac{\lambda_1(t_{in})}{\lambda_1(t_{1})} \frac{\lambda_3(t_1)}{\lambda_3(t_3)} \mathbb{E}\left[D_{t_3}^{(v_3)} ~ | ~ \mathcal{F}_{\hat{t}}  \right],
\end{align*}
where $\hat{t} \in [t_0,t_{in}]$. The total inflow $u(t_{in})$ is then given by summing up the individual inflow shares, i.e. $u(t_{in}) = u_2(t_{in}) + u_3(t_{in})$. Again, choosing the velocity functions $\lambda_i$ constant, we end up in the special case of linear advection on networks. The approach can be easily extended to larger arbitrary tree networks by calculating the individual inflow shares of the demand nodes which are basically larger 1-1 networks. The overall inflow results in the sum of the individual inflow shares.

So far, we only consider the optimal inflow using backward calculation. The discussion on the distribution parameters at the nodes for the forward calculation starting with the inflow will be postponed to Section $\ref{cha.3}$ because it might be additionally dependent on the degree of information that is provided by the system.

\subsection{Linear transport with time-dependent velocity and damping function}
In this section we extend the framework by an additional damping term $g^{(i)}(z,t) = \mu_i(t)z$, where $\mu_i \in L^1([t_0,T])$ is chosen to be a non-negative function. The flux function on arc $i$ is chosen as before $f^{(i)}(z,t) = \lambda_i(t) z$ with $\lambda_i(\cdot)$ strictly positive. The damping reflects a loss in the transported quantity over time, which may be due to some physical property as for instance friction or electrical resistance. We allow to have an individual shape of the damping function for any arc which may additional be dependent on the time. 

\subsubsection{1-1 Network}
\label{1-1damp}
For a 1-1 network we show how to calculate the optimal input at the initial node taking the additional damping into consideration. We work with the network from Figure $\ref{network1-1}$ and use the same notation as before, meaning $t_2$ is the time in which a unit injected at $t_{in}$ reaches the demand node $v_2$ and $t_1$ the time when it reaches node $v_1$. Under the assumptions of Lemma $\ref{Lem41}$, the values for $t_1,t_2$ are given by $(\ref{intermediateTimes2})$ and are not influenced by the damping term. 

To incorporate the damping into the calculation of the optimal input, we consider again the characteristics from the PDE in equation $(\ref{PDEwD})$ with $g(z^{(i)},t)= 0$, i.e. the situation without damping, for an arc $i$ and denote them by $(x(t),t)$. Along this curve the density $z^{(i)}$ stays constant, meaning that $\frac{d}{dt} z^{(i)}(x(t),t) = 0$. If we now add a damping term, the density is reduced proportional to $\mu_i(t)z^{(i)}(x(t),t)$ which can be formulated by the ordinary differential equation
\begin{align*}
    \frac{d}{dt}z^{(i)}(x(t),t) = - \mu_i(t) z^{(i)}(x(t),t).
\end{align*}
Adding the injection information into the arc at time $t_{i-1}$ the initial value problem
\begin{align*}
    \frac{d}{dt}z^{(i)}(x(t),t) = - \mu_i(t) z^{(i)}(x(t),t),~~~ z^{(i)}(x(t_{i-1}),t_{i-1}) = z_{t_{i-1}}
\end{align*}
has the unique solution
\begin{align}
\label{dampODE}
    z(x(t),t) = z_{t_{i-1}}e^{-\int_{t_{i-1}}^t \mu_i(s)ds}.
\end{align}
The integral in the exponential function stays bounded since $\mu_i \in L^1([t_0,T])$.
Applying this result to the 1-1 case the initial injection at node $v_0$ has to be scaled up such that after the damping on both arcs the optimal flux reaches the demand node. 
Backward calculating the inflow-outflow relation, using equation $(\ref{dampODE})$ and the previous assumption on the flux conservation at the nodes, we get 
\begin{align*}
    f^{(2)}(z^{(2)}(1,t_2),t_2) &= z^{(2)}(1,t_2) \lambda_2(t_2) = z^{(2)}(0,t_1)e^{-\int_{t_1}^{t_2} \mu_2(s)ds} \lambda_2(t_2)\\
    &= f^{(2)}(z^{(2)}(0,t_1),t_1) \frac{\lambda_2(t_2)}{\lambda_2(t_1)} e^{-\int_{t_1}^{t_2}\mu_2(s)ds} \\
    &= f^{(1)}(z^{(1)}(1,t_1),t_1) \frac{\lambda_2(t_2)}{\lambda_2(t_1)} e^{-\int_{t_1}^{t_2}\mu_2(s)ds}\\
    &= z^{(1)}(1,t_1) \frac{\lambda_1(t_1)\lambda_2(t_2)}{\lambda_2(t_1)} e^{-\int_{t_1}^{t_2}\mu_2(s)ds} \\
    &= z^{(1)}(0,t_{in})e^{-\int_{t_{in}}^{t_1} \mu_1(s)ds} \frac{\lambda_1(t_1)\lambda_2(t_2)}{\lambda_2(t_1)} e^{-\int_{t_1}^{t_2}\mu_2(s)ds} \\
    &= f^{(1)}(z^{(1)}(0,t_{in}),t_{in}) \frac{\lambda_1(t_1)\lambda_2(t_2)}{\lambda_1(t_{in})\lambda_2(t_1)} e^{-\int_{t_{in}}^{t_1} \mu_1(s)ds} e^{-\int_{t_1}^{t_2}\mu_2(s)ds}.
\end{align*}

Since the outflow should match the conditional expectation of the demand, we find the optimal inflow at node $v_0$ by
\begin{align}
\label{1-2infl}
    u(t_{in}) =  \frac{\lambda_1(t_{in})}{\lambda_1(t_{1})} \frac{\lambda_2(t_1)}{\lambda_2(t_2)} \mathbb{E}\left[D_{t_2}^{(v_2)} ~ | ~ \mathcal{F}_{\hat{t}}  \right] e^{\int_{t_{in}}^{t_1} \mu_1(s) ds}  e^{\int_{t_1}^{t_2} \mu_2(s) ds}.
\end{align}

The solution procedure can iteratively be extended to larger 1-1 networks by multiplying factors of the type $\frac{\lambda_i(t_{i-1})}{\lambda_i(t_i)}e^{\int_{t_{i-1}}^{t_i}\mu_i(s)ds}$ to the right hand side of $(\ref{1-2infl})$.

\subsubsection{1-2 Network}
Finally, we transfer the results from the 1-1 network to the 1-2 network from Figure $\ref{network1-2}$ combining the techniques from Section $\ref{1-2noDamp}$ and $\ref{1-1damp}$. Assuming that the conditions in Lemma $\ref{Lem41}$ hold true the times $t_i, ~i=1,2,3$ are given as in equation $(\ref{intermediateTimes})$. For both of the demand nodes, we calculate the optimal injection backwards focusing on the relevant subnetworks which reduce to two 1-1 networks. Equation $(\ref{1-2infl})$ then leads to the individual optimal inflows of
\begin{align*}
    u_2(t_{in}) &=  \frac{\lambda_1(t_{in})}{\lambda_1(t_{1})} \frac{\lambda_2(t_1)}{\lambda_2(t_2)} \mathbb{E}\left[D_{t_2}^{(v_2)} ~ | ~ \mathcal{F}_{\hat{t}}  \right]e^{\int_{t_{in}}^{t_1} \mu_1(s) ds} e^{\int_{t_1}^{t_2} \mu_2(s) ds}\\
    u_3(t_{in}) &=  \frac{\lambda_1(t_{in})}{\lambda_1(t_{1})} \frac{\lambda_3(t_1)}{\lambda_3(t_3)} \mathbb{E}\left[D_{t_3}^{(v_3)} ~ | ~ \mathcal{F}_{\hat{t}}  \right]e^{\int_{t_{in}}^{t_1} \mu_1(s) ds} e^{\int_{t_1}^{t_3} \mu_3(s) ds}.
\end{align*}
The total optimal inflow $u(t_{in})$ at node $v_0$ is then given by $u(t_{in}) = u_2(t_{in}) + u_3(t_{in})$.

\section{Different control strategies based on the degree of information}
\label{cha.3}
This section deals with different control strategies based on different information levels. We are able to adjust two influencing factors in the supply network. On the one hand, we can control the inflow at node $v_0$ (represented by equation (\ref{boundaryData}) in the optimization problem) and on the other hand we can rearrange the distribution at each inner node (represented by (\ref{fluxCons1}) and (\ref{fluxCons2})). An additional adjustment at the inner nodes seems reasonable if there is more recent demand information available than at the time of the injection. The choice of the optimal injection has been mainly discussed in Section $\ref{cha4}$. It remains to determine the optimal distribution parameters $\alpha_{i,k}(t)$. Before we investigate the different scenarios further, we introduce some notation.

Let us define a function $\tilde{p}$ that maps a given node to the direct predecessor node. Similarly, the function $\tilde{q}$ maps a given arc $i$ to the directly preceding arc. The function $\tilde{c}$ maps a node $v_i$ to all \textit{demand}-nodes which are successors of $v_i$. The set $J^{out}_{v_i}$ is defined to be the set of all directly outgoing arcs of node $v_i$. 
Additionally, the function $\tilde{t}$ assigns, given two nodes $v_i$, $v_j$ and a time $t$, the time $\tilde{t}(v_i,v_j,t)$ at which the electricity that leaves node $v_i$ at time $t$ reaches node $v_j$. Last, denote $\eta(v_i,v_j)$ to be the sequence of all network arcs leading from the inner node $v_i$ to the demand node $v_j$.

\subsection{Model setting 1: single demand update}
In model setting 1 (\textbf{MS1}) we consider the least information about demand that is possible. We assume that there is only one demand update right at the beginning of the time period at $t_0$. Further developments of the demand processes are not taken into consideration for the injection at the source or the distribution at the inner nodes. 
Therefore, according to Section $\ref{cha4}$ the optimal inflow is given by
\begin{align}
\begin{split}
\label{inflowMS1}
    u(t_{in}) &= \sum_{\lbrace i ~: ~v_i \in C\rbrace} \mathbb{E}\left[D_{\tilde{t}(v_0,v_i,t_{in})}^{(v_i)} | \mathcal{F}_{t_0} \right] \frac{\lambda_1(t_{in})}{\lambda_1(t_1)}e^{\int_{t_{in}}^{\tilde{t}(v_0,v_1,t_{in})}\mu_i(s)ds}\\
    &~~~\dot~ \prod_{l \in \eta(v_1,v_i)} \frac{\lambda_{l}(\tilde{t}(v_0,\tilde{p}(v_l),t_{in}))}{\lambda_{l}(\tilde{t}(v_0,v_l,t_{in}))}e^{\int_{\tilde{t}(v_0,\tilde{p}(v_l),t_{in})}^{\tilde{t}(v_0,v_l,t_{in})}\mu_l(s)ds}.
\end{split}
\end{align}
In contrast to Section $\ref{cha4}$, where we focus on 1-1 and 1-2 networks only, we now generalized the idea of building sub 1-1 networks for any demand node by introducing the set $\eta(v_1,v_i)$. 

We note that also the distribution parameters at the inner nodes are chosen optimally with respect to the knowledge at $t_0$. Since we have already calculated the optimal inflow shares for any demand node, they can be reused to determine the optimal distribution parameters. For any inner node $v_i$ this can be done by allocating all (demand) nodes included in the set $\tilde{c}(v_i)$ to the corresponding outgoing arc $k$ of $v_i$. Then, the distribution parameter for the share of ingoing flux moving from arc $i$ to arc $k$ at time $t$ is given by
\begin{align*}
 \alpha_{i,k}(t)= \frac{\sum_{v_q \in {\tilde{c}(v_k)}} \mathbb{E}\left[D_{\tilde{t}(v_i,v_q,t)}^{(v_q)} ~\big|~ \mathcal{F}_{t_0}\right]\prod_{l \in \eta(v_i,v_q)} \frac{\lambda_{l}(\tilde{t}(v_i,\tilde{p}(v_l),t))}{\lambda_{l}(\tilde{t}(v_i,v_l,t))}e^{\int_{\tilde{t}(v_i,\tilde{p}(v_l),t)}^{\tilde{t}(v_i,v_l,t)}\mu_l(s)ds}}{\sum_{v_r \in \tilde{c}(v_i)}  \mathbb{E}\left[D_{\tilde{t}(v_i,v_r,t)}^{(v_r)} ~ \big|~ \mathcal{F}_{t_0}\right] \prod_{l \in \eta(v_i,v_r)} \frac{\lambda_{l}(\tilde{t}(v_i,\tilde{p}(v_l),t))}{\lambda_{l}(\tilde{t}(v_i,v_l,t))}e^{\int_{\tilde{t}(v_i,v_l,t)}^{\tilde{t}(v_i,v_l,t)}\mu_l(s)ds}}.
\end{align*}
The distribution parameters are mainly driven by the expected demands conditioned on the state of the system at initial time $t_0$. In contrast to the calculations for the inflow at the source node in $(\ref{inflowMS1})$, the terms which stem from the flux conservation property and the damping are now starting in node $v_i$ and not in the source node. Since $\tilde{c}(v_k) \subset \tilde{c}(v_i)$ if arc $k$ is a successor of arc $i$ and all quantities are non-negative, it holds $\alpha_{i,k}(t) \in [0,1]$. Additionally, it holds that, due to the tree network structure, $\tilde{c}(v_i)$ can be comprised of the disjoint union $\biguplus\limits_{k \in J_{v_i}^{out}} \tilde{c}(v_k)$ which directly leads to the property
\begin{align}
\label{sumAlpha}
    \sum_{k \in J_{v_i}^{out}} \alpha_{i,k}(t) = 1
\end{align}
for any inner node $v_i$ and $t \in [t_0,T]$. 

\subsection{Model setting 2: multiple time-delayed demand updates}
In model setting 2 (\textbf{MS2}) we allow for regular demand updates to improve the accuracy of the injection at the source node $v_0$. The update times are given by a sequence $\lbrace \hat{t}_1,\dots, \hat{t}_k \rbrace$ with $t_0= \hat{t}_1 < \dots <\hat{t}_k \leq T$. We denote by $\tilde{t}^{-1}$ the inverse of the function $\tilde{t}$ with respect to the time argument such that $\tilde{t}^{-1}(v_0,v_i,s)$ returns the time at which the unit that reaches node $v_i$ at time $s$ has been injected in the system at $v_0$. The notation of $\lfloor s \rfloor_{\hat{t_j}}$ assigns the largest update time which is smaller than $s$. To account for the update times, we now consider a family of optimization problems of the following type:
\begin{align}
    \begin{split}
    \label{networkOCP}
        \underset{u \in L^2}{\min} & \sum_{\lbrace i~:~ v_i \in C \rbrace} \int_{\min \lbrace \tilde{t}(v_0,v_i,\hat{t}_j), T\rbrace}^{\min \lbrace\tilde{t}(v_0,v_i,\hat{t}_{j+1}), T \rbrace} \mathbb{E}\left[\left(D_s^{(v_i)} - f^{(i)}(z^{(i)}(1,s),s)\right)^2 ~\Big|~ \mathcal{F}_{\hat{t}_j}\right] ds\\
        \text{s.t.}~~~~ & z_t^{(i)}(x,t) + f^{(i)}(z^{(i)}(x,t),t)_x + g^{(i)}(z^{(i)}(x,t),t) = 0, ~~~~ \forall i \text{ s.t. } v_i \in J \cup C\\
        & z^{(i)}(x,\hat{t}_j) = z_{\text{old}}^{(i)}(x,\hat{t}_j), ~~~~ \forall i \text{ s.t. } v_i \in J \cup C\\
        & z^{(1)}(0,t)\lambda_1(t) = u(t) \\
        &f^{(k)}(z^{(k)}(0,t),t) = \alpha_{i,k}(t) f^{(i)}(z^{(i)}(1,t),t),~~~~ \forall i \text{ s.t. } v_i \in J, ~k \text{ outgoing arc of } v_i\\
        &\sum_{k \in J_{v_i}^{out}} \alpha_{i,k}(t) = 1,~~~~ \forall i \text{ s.t. } v_i \in J\\
        & x \in (0,1), ~~t \in [\hat{t}_j, \hat{t}_{j+1}].
    \end{split}
\end{align}
In addition to the initial condition at $t_0$, we have to set initial conditions at $\hat{t}_j$ for each optimization problem, which are naturally given by the system state of the predecessor problem at its terminal time. By $z_{\text{old}}^{(i)}(x,\hat{t}_j)$ we denote the state of the densities on the network arcs in the system at time $\hat{t}_j$. 

Compared to MS1, the inflow choice is not based on the demand information at $t_0$, but on the information that is available when the last demand update happened. Therefore, the optimal inflow at some time $t_{in} \in [\hat{t}_j, \hat{t}_{j+1})$ is given by
\begin{align}
\begin{split}
\label{inflowMS2}
    u(t_{in}) &= \sum_{\lbrace i ~: ~v_i \in C\rbrace} \mathbb{E}\left[D_{\tilde{t}(v_0,v_i,t_{in})}^{(v_i)} | \mathcal{F}_{\hat{t}_j} \right] \frac{\lambda_1(t_{in})}{\lambda_1(t_1)}e^{\int_{t_{in}}^{\tilde{t}(v_0,v_1,t_{in})}\mu_i(s)ds}\\
    &~~~\cdot \prod_{l \in \eta(v_1,v_i)} \frac{\lambda_{l}(\tilde{t}(v_0,\tilde{p}(v_l),t_{in}))}{\lambda_{l}(\tilde{t}(v_0,v_l,t_{in}))}e^{\int_{\tilde{t}(v_0,\tilde{p}(v_l),t_{in})}^{\tilde{t}(v_0,v_l,t_{in})}\mu_l(s)ds}.
\end{split}
\end{align}

At the inner nodes, there is also an update procedure for the distribution parameters. For MS2 we follow the idea that the distribution parameters for particular units are determined when they enter the supply network at node $v_0$. Therefore, we always use the demand information which is available at the time a particular unit has been injected into the system. Then, the optimal distribution parameter for the share of the flux moving from arc $i$ to arc $k$ at node $v_i$ and time $t$ is given by
\begin{align*}
 \alpha_{i,k}(t)= \dfrac{\sum_{v_q \in {\tilde{c}(v_k)}} \mathbb{E}\left[D_{\tilde{t}(v_i,v_q,t)}^{(v_q)} ~\big|~ \mathcal{F}_{\lfloor \tilde{t}^{-1}(v_0,{v_i},t) \rfloor_{\hat{t_j}}}\right]\prod_{l \in \eta(v_i,v_q)} \frac{\lambda_{l}(\tilde{t}(v_i,\tilde{p}(v_l),t))}{\lambda_{l}(\tilde{t}(v_i,v_l,t))}e^{\int_{\tilde{t}(v_i,\tilde{p}(v_l),t)}^{\tilde{t}(v_i,v_l,t)}\mu_l(s)ds}}{\sum_{v_r \in \tilde{c}(v_i)}  \mathbb{E}\left[D_{\tilde{t}(v_i,v_r,t)}^{(v_r)} ~ \big|~ \mathcal{F}_{\lfloor \tilde{t}^{-1}(v_0,{v_i},t) \rfloor_{\hat{t_j}}}\right] \prod_{l \in \eta(v_i,v_r)} \frac{\lambda_{l}(\tilde{t}(v_i,\tilde{p}(v_l),t))}{\lambda_{l}(\tilde{t}(v_i,v_l,t))}e^{\int_{\tilde{t}(v_i,v_l,t)}^{\tilde{t}(v_i,v_l,t)}\mu_l(s)ds}}.
\end{align*}
The main difference to MS1 is that we explicitly have to calculate the time $\tilde{t}^{-1}(v_0,{v_i},t)$ at which we have the latest information about the demand.
The properties of $\alpha_{i,k}(t)$ presented in MS1 apply also for MS2. If we choose $t_0$ to be the update time, we end up with the distribution parameters developed in MS1.

\subsection{Model setting 3: multiple instantaneous demand updates}
Similar to MS2, the last model setting 3 (\textbf{MS3}) also allows for regular demand updates at $\lbrace \hat{t}_1,\dots, \hat{t}_k \rbrace$ with $t_0= \hat{t}_1 < \dots <\hat{t}_k \leq T$. Therefore, we again consider a family of optimization problems presented in equation ($\ref{networkOCP}$). Also the choices of the optimal injection at the source node $v_0$ from MS2 in equation ($\ref{inflowMS2}$) remain unchanged because we assume the injection is always chosen optimally with respect to the latest demand update. The new aspect of MS3 is that also at the inner nodes the flux distribution is arranged optimally with respect to the latest demand information and not only with respect to the latest demand information when the particular unit entered the supply network. 
Therefore, the distribution parameters at the nodes are not given by the injection shares, but we have to consider an additional family of optimization problems for each inner node. The ingoing flux is considered to be given and flux conservation is required. Based on the choice of the distribution parameters $\alpha_{i,k}(t)$ we aim to minimize for a given inner node $v_i$ and a time $t$ the expected quadratic deviation of the difference of demand and supply, i.e.
\begin{align}
    \begin{split}
        \label{optContMS3}
        &\underset{\alpha_{i,k}(t) \in [0,1]}{\min} \sum_{\lbrace r:v_r \in \tilde{c}(v_i) \rbrace} \mathbb{E}\left[ \left(D_{\tilde{t}(v_i,v_r,t)}^{(v_r)} - f^{(r)}(z^{(r)}(1,\tilde{t}(v_i,v_r,t)),t)  \right)^2   ~|~ \mathcal{F}_{\lfloor t \rfloor_{\hat{t}_j}}\right]\\
        \text{s.t. }& f^{(k)} (z^{(k)}(0,t),t) = \alpha_{i,k}(t) f^{(i)}(z^{(i)}(1,t),t), k\in J^{out}_{v_i}\\
        &\sum_{k \in J_{v_i}^{out}} \alpha_{i,k}(t) = 1\\
        &f^{(k)}(z^{(k)}(0,t),t) = \sum_{\lbrace r:v_r \in \tilde{c}(v_k)\rbrace} f^{(r)}(z^{(r)}(1,\tilde{t}(v_i,v_r,t)),t) \\
        &\hspace{3.2cm}\cdot\prod_{l \in \eta(v_i,v_r)} \frac{\lambda_{l}(\tilde{t}(v_i,\tilde{p}(v_l),t))}{\lambda_{l}(\tilde{t}(v_i,v_l,t))}e^{\int_{\tilde{t}(v_i,\tilde{p}(v_l),t)}^{\tilde{t}(v_i,v_l,t)}\mu_l(s)ds}, ~~k \in J^{out}_{v_i}.
    \end{split}
\end{align}
In Section $\ref{cha4}$ we have explained that the flux at the demand node is related to the fluxes at the inner nodes taking into account the damping and the weighted densities at the inner nodes. Using these relations and introducing the following two definitions for the sake of better clarity, we can rewrite $(\ref{optContMS3})$ where we additionally choose the supply at the demand nodes to be the optimization variable and replace the distribution parameters for a second.
\begin{align}
    \begin{split}
        \label{optContMS3new}
        &\underset{m_i^r(t)}{\min} \sum_{\lbrace r:v_r \in \tilde{c}(v_i)\rbrace} \mathbb{E}\left[ \left(D_{\tilde{t}(v_i,v_r,t)}^{(v_r)} - m_i^r(t)  \right)^2   ~|~ \mathcal{F}_{\lfloor t \rfloor_{\hat{t}_j}}\right]\\
        \text{s.t. }& f^{(i)}(z^{(i)}(1,t),t) = \sum_{\lbrace r:v_r \in \tilde{c}(v_i)\rbrace} m_i^r(t) \gamma_i^r(t),
    \end{split}
\end{align}
where
\begin{align}
    \gamma_i^r(t) &:= \prod_{l \in \eta(v_i,v_r)} \frac{\lambda_{l}(\tilde{t}(v_i,\tilde{p}(v_l),t))}{\lambda_{l}(\tilde{t}(v_i,v_l,t))}e^{\int_{\tilde{t}(v_i,\tilde{p}(v_l),t)}^{\tilde{t}(v_i,v_l,t)}\mu_l(s)ds}\\
    m_i^r(t) &:= f^{(r)}(z^{(r)}(1,\tilde{t}(v_i,v_r,t)),t).
\end{align}
We end up with a family of optimization problems with one constraint. The constraint ensures flux conservation at the inner node $v_i$.
The following Lemma provides an optimality result on how the distribution of the flow at an arbitrary inner node $v_i$ should be arranged according to the latest demand information at $\hat{t}_j$. 

\begin{lemma}
\label{optMS3Lemma}
Fix an inner node $v_i$. For $t \in [\hat{t}_j,\hat{t}_{j+1}]$ consider the optimization problem
\begin{align*}
    &\underset{m_i^r(t)}{\min} \sum_{v_r \in \tilde{c}(v_i)} \mathbb{E} \left[\left( D_{\tilde{t}(v_i,v_r,t)}^{(v_r)} - m_i^r(t)\right)^2 | \mathcal{F}_{\hat{t}_j} \right]\\
    & \text{s.t. } f_{in}^{(i)}(t):= f^{(i)}(z^{(i)}(1,t),t) = \sum_{\lbrace r:v_r \in \tilde{c}(v_i)\rbrace} m_i^r(t) \gamma_i^r(t),
\end{align*}
where $\hat{t}_j$ is the time until which information is available. The optimal choices for $m_i^r(t)$, for which $v_{r} \in  \tilde{c}(v_i)$, are given by
\begin{align*}
    m_i^{r}(t) &= \mathbb{E}\left[D_{\tilde{t}(v_i,v_{r},t)}^{(v_{r})} | \mathcal{F}_{\hat{t}_j}\right]\\
&~+\frac{\gamma_i^r(t)}{\sum_{v_c \in \tilde{c}(v_i)} \left(\gamma_i^c(t)\right)^2} \left(f_{in}^{(i)}(t) - \sum_{v_c \in \tilde{c}(v_i)} \gamma_i^c(t) \mathbb{E}\left[D_{\tilde{t}(v_i,v_c,t)}^{(v_c)} | \mathcal{F}_{\hat{t}_j} \right]   \right).
\end{align*}
\end{lemma}
\begin{proof}
We set up the Lagrangian function for $\xi \in \mathbb{R}$: 
\begin{align*}
    \mathcal{L}(m_i^{r_1}(t),\dots, m_i^{{r_{|C|}}}(t), \xi) &= \sum_{v_{r_l} \in \tilde{c}(v_i)} \mathbb{E}\left[\left(D_{\tilde{t}(v_i,v_{r_l},t)}^{(v_r)} - m_i^{r_l}(t) \right)^2 | \mathcal{F}_{\hat{t}_j}\right] \\
    &~~- \xi \left( f_{in}^{(i)}(t) - \sum_{v_{r_l} \in \tilde{c}(v_i)} \gamma_i^{r_l}(t) \cdot m_i^{r_l}(t)\right),
\end{align*}
where $|C|$ denotes the number of demand nodes in $\tilde{c}(v_i)$. The partial derivatives of $\mathcal{L}$ are given by
\begin{align}
    \frac{\partial \mathcal{L}}{\partial m_i^{r_q}(t)}\left(m_i^{r_1}(t),\dots, m_i^{r_{|C|}}(t), \xi\right) &= -2 \mathbb{E}\left[D_{\tilde{t}(v_i,v_{r_q},t)}^{(v_{r_q})} | \mathcal{F}_{\hat{t}_j}\right] + 2 m_i^{r_q}(t)+ \xi \gamma_{i}^{r_q}(t) ,~~~ v_{r_q} \in \tilde{c}(v_i) \label{AblL1}\\
    \frac{\partial \mathcal{L}}{\partial \xi}\left(m_i^{r_1}(t),\dots, m_i^{r_{|C|}}(t), \xi \right) &=  f_{in}^{(i)}(t) - \sum_{v_{r_l} \in \tilde{c}(v_i)} \gamma_{i}^{r_l}(t) \cdot m_i^{r_l}(t) \label{AblL2}.
\end{align}
Setting $(\ref{AblL1})$ equal zero, we obtain the following two expressions
\begin{align}
    m_i^{r_q}(t) &= \mathbb{E}\left[D_{\tilde{t}(v_i,v_{r_q},t)}^{(v_{r_q})} | \mathcal{F}_{\hat{t}_j}  \right] - \frac{\xi}{2}\gamma_{i}^{{r_q}}(t) \label{AblL3}\\
    \frac{\xi}{2} &= \frac{1}{\gamma_{i}^{r_q}(t)}  \left(\mathbb{E}\left[D_{\tilde{t}(v_i,v_{r_q},t)}^{(v_{r_q})} | \mathcal{F}_{\hat{t}_j}  \right] - m_i^{r_q}(t) \right)\label{AblL4}.
\end{align}
Using equation $(\ref{AblL3})$ for an arbitrary but fixed $v_{r_1} \in \tilde{c}(v_i)$ and plugging in $(\ref{AblL4})$ for arbitrary but fixed $v_{r_2} \in \tilde{c}(v_i), ~v_{r_1}\neq v_{r_2}$, we get
\begin{align}
    m_i^{r_1}(t) = \mathbb{E} \left[D_{\tilde{t}(v_i,v_{r_1},t)}^{(v_{r_1})} | \mathcal{F}_{\hat{t}_j}  \right] - \frac{\gamma_{i}^{r_1}(t)}{\gamma_{i}^{r_2}(t)} \mathbb{E} \left[D_{\tilde{t}(v_i,v_{r_2},t)}^{(v_{r_2})} | \mathcal{F}_{\hat{t}_j}  \right] + \frac{\gamma_{i}^{r_1}(t)}{\gamma_{i}^{r_2}(t)} m_i^{r_2}(t). \label{AblL5}
\end{align}
Setting $(\ref{AblL2})$ equal zero and solving for $m_i^{r_2}(t)$ we get together with plugging in $m_i^{r_l}(t)$ from $(\ref{AblL3})$
\begin{align*}
    \gamma_{i}^{r_2}(t) m_i^{r_2}(t) &= f_{in}^{(i)}(t) - \sum_{v_{r_l} \in \tilde{c}(v_i), v_{r_l}\neq v_{r_2}} \gamma_{i}^{r_l}(t)  \mathbb{E} \left[D_{\tilde{t}(v_i,v_{r_l},t)}^{(v_{r_l})} | \mathcal{F}_{\hat{t}_j} \right] + \sum_{v_{r_l} \in \tilde{c}(v_i), v_{r_l}\neq v_{r_2}} \left(\gamma_{i}^{r_l}(t) \right)^2 \frac{\xi}{2}\\
    & \overset{(\ref{AblL4})}{=} f_{in}^{(i)}(t) - \sum_{v_{r_l} \in \tilde{c}(v_i), v_{r_l}\neq v_{r_2}} \gamma_{i}^{r_l}(t)  \mathbb{E} \left[D_{\tilde{t}(v_i,v_{r_l},t)}^{(v_{r_l})} | \mathcal{F}_{\hat{t}_j} \right] \\
    &~~~+ \sum_{v_{r_l} \in \tilde{c}(v_i), v_{r_l}\neq v_{r_2}}  \frac{\left(\gamma_{i}^{r_l}(t) \right)^2}{\gamma_{i}^{r_1}(t)}  \left(\mathbb{E}\left[D_{\tilde{t}(v_i,v_{r_1},t)}^{(v_{r_1})} | \mathcal{F}_{\hat{t}_j}  \right] - m_i^{r_1}(t) \right)
    \end{align*}
    Solving this equation for $m_i^{r_2}(t)$ yields
    \begin{align*}
    m_i^{r_2}(t) &= \frac{1}{\gamma_{i}^{r_2}(t)}\Bigg(f_{in}^{(i)}(t) - \sum_{v_{r_l} \in \tilde{c}(v_i), v_{r_l}\neq v_{r_2}} \gamma_{i}^{r_l}(t)  \mathbb{E} \left[D_{\tilde{t}(v_i,v_{r_l},t)}^{(v_{r_l})} | \mathcal{F}_{\hat{t}_j} \right] \\
    &~~~+ \sum_{v_{r_l} \in \tilde{c}(v_i), v_{r_l}\neq v_{r_2}}  \frac{\left(\gamma_{i}^{r_l}(t) \right)^2}{\gamma_{i}^{r_1}(t)}  \left(\mathbb{E}\left[D_{\tilde{t}(v_i,v_{r_1},t)}^{(v_{r_1})} | \mathcal{F}_{\hat{t}_j}  \right] - m_i^{r_1}(t) \right)  \Bigg).
\end{align*}
Plugging $m_i^{r_2}$ into $(\ref{AblL5})$ we get
\begin{align*}
    m_i^{r_1}(t) &= \mathbb{E} \left[D_{\tilde{t}(v_i,v_{r_1},t)}^{(v_{r_1})} | \mathcal{F}_{\hat{t}_j}  \right] - \frac{\gamma_{i}^{r_1}(t)}{\gamma_{i}^{r_2}(t)} \mathbb{E} \left[D_{\tilde{t}(v_i,v_{r_2},t)}^{(v_{r_2})} | \mathcal{F}_{\hat{t}_j}  \right] + \frac{\gamma_{i}^{r_1}(t)}{(\gamma_{i}^{r_2}(t))^2} f_{in}^{(i)}(t) \\
    &~- \sum_{v_{r_l} \in \tilde{c}(v_i), v_{r_l}\neq v_{r_2}} \frac{\gamma_{i}^{r_1}(t)\gamma_{i}^{r_l}(t)}{\left(\gamma_{i}^{r_2}(t)\right)^2}\mathbb{E} \left[D_{\tilde{t}(v_i,v_{r_l},t)}^{(v_{r_l})} | \mathcal{F}_{\hat{t}_j}  \right] \\
    &~+ \sum_{v_{r_l} \in \tilde{c}(v_i), v_{r_l}\neq v_{r_2}} \frac{\left(\gamma_{v_i}^{r_l}(t)\right)^2}{\left(\gamma_{i}^{r_2}(t)\right)^2}\left(\mathbb{E}\left[D_{\tilde{t}(v_i,v_{r_1},t)}^{(v_{r_1})} | \mathcal{F}_{\hat{t}_j} \right] - m_i^{r_1}(t)  \right)\\
    \end{align*}
    Next, we collect all terms of $m_i^{r_1}(t)$ on the left side. Additionally, we expand the factors of the conditional expectations such that they are gathered in one summation over all demand nodes.
\begin{align*}
    m_i^{r_1}(t)  \left(\sum_{v_{r_l} \in \tilde{c}(v_i)} \left(\gamma_{i}^{r_l}(t)\right)^2 \right) &= \mathbb{E} \left[D_{\tilde{t}(v_i,v_{r_1},t)}^{(v_{r_1})}   | \mathcal{F}_{\hat{t}_j}\right]\left(\sum_{v_{r_l} \in \tilde{c}(v_i)} \left(\gamma_{v_i}^{r_l}(t)\right)^2 \right) + \gamma_{i}^{r_1}(t)f_{in}^{(i)}(t) \\
    &~- \sum_{v_{r_l} \in \tilde{c}(v_i)}\gamma_{i}^{r_1}(t)\gamma_{i}^{r_l}(t) \mathbb{E}\left[D_{\tilde{t}(v_i,v_{r_l},t)}^{(v_{r_l})} | \mathcal{F}_{\hat{t}_j}  \right] \\
    \Leftrightarrow m_i^{r_1}(t)  &= \mathbb{E} \left[D_{\tilde{t}(v_i,v_{r_1},t)}^{(v_{r_1})}   | \mathcal{F}_{\hat{t}_j}\right]  \\
    &+ \frac{\gamma_{i}^{r_1}(t)}{\sum_{v_{r_l} \in \tilde{c}(v_i)} \left(\gamma_{i}^{r_l}(t)\right)^2} \left(f_{in}^{(i)}(t) - \sum_{v_{r_l} \in \tilde{c}(v_i)} \gamma_{i}^{r_l}(t) \mathbb{E}\left[D_{\tilde{t}(v_i,v_{r_l},t)}^{(v_{r_l})} | \mathcal{F}_{\hat{t}_j} \right]   \right).
\end{align*}
The optimal inflow from arc $i$ into arc $k$ is then given by
\begin{align*}
    f^{(k)} (z^{(k)}(0,t),t) &= \sum_{v_q \in \tilde{c}(v_{k})} \left( \gamma_i^{q}(t) \cdot \mathbb{E} \left[D_{\tilde{t}(v_i,v_q,t)}^{(v_q)}   | \mathcal{F}_{\hat{t}_j}\right] \right) \\
    &~~+ \frac{\sum_{v_q \in \tilde{c}(v_{k})}\left(\gamma_{i}^{q}(t)\right)^2}{\sum_{v_{r} \in \tilde{c}(v_i)} \left(\gamma_{i}^{r}(t)\right)^2} \left(f_{in}^{(i)}(t) - \sum_{v_{r} \in \tilde{c}(v_i)} \gamma_{i}^{r}(t) \mathbb{E}\left[D_{\tilde{t}(v_i,v_{r},t)}^{(v_{r})} | \mathcal{F}_{\hat{t}_j} \right]   \right). \qedhere
\end{align*}
\end{proof}

Using this result we can directly deduce the values for the optimal distribution parameter $\alpha_{i,k}(t)$ of the flux directed from arc $i$ to arc $k$ at time $t$ in MS3 by
\newpage
{\small
\begin{align*}
    &\alpha_{i,k}(t) =\\ &\dfrac{\sum\limits_{v_q \in \tilde{c}(v_{k})} \left( \gamma_i^{q}(t) \cdot \mathbb{E} \left[D_{\tilde{t}(v_i,v_q,t)}^{(v_q)}   | \mathcal{F}_{\hat{t}_j}\right] \right) + \dfrac{\sum\limits_{v_q \in \tilde{c}(v_{k})}\left(\gamma_{i}^{q}(t)\right)^2}{\sum\limits_{v_{r} \in \tilde{c}(v_i)} \left(\gamma_{i}^{r}(t)\right)^2} \left(f_{in}^{(i)}(t) - \sum\limits_{v_{r} \in \tilde{c}(v_i)} \gamma_{i}^{r}(t) \mathbb{E}\left[D_{\tilde{t}(v_i,v_{r},t)}^{(v_{r})} | \mathcal{F}_{\hat{t}_j} \right]   \right)}{\sum\limits_{v_r \in \tilde{c}(v_{i})} \left( \gamma_i^{r}(t) \cdot \mathbb{E} \left[D_{\tilde{t}(v_i,v_r,t)}^{(v_q)}   | \mathcal{F}_{\hat{t}_j}\right] \right) + \dfrac{\sum\limits_{v_q \in \tilde{c}(v_{k})}\left(\gamma_{i}^{q}(t)\right)^2}{\sum\limits_{v_{r} \in \tilde{c}(v_i)} \left(\gamma_{i}^{r}(t)\right)^2} \left(f_{in}^{(i)}(t) - \sum\limits_{v_{r} \in \tilde{c}(v_i)} \gamma_{i}^{r}(t) \mathbb{E}\left[D_{\tilde{t}(v_i,v_{r},t)}^{(v_{r})} | \mathcal{F}_{\hat{t}_j} \right]   \right)}.
\end{align*} }
The property from equation $(\ref{sumAlpha})$ that the distribution parameters sum up to 1 still holds true in MS3. But we face the drawback, that we cannot guarantee anymore that $\alpha_{i,k}(t) \in [0,1]$, especially in environments in which demand is highly volatile and the inflow is low. Negative distribution parameters lead to a negative inflow into an arc which represents an incorrect and not meaningful solution. This issue can be avoided by setting potential negative distribution parameters to 0 and add the remaining shares proportional to the other outgoing arcs. In the special case in which the inflow into node $v_i$ matches the expected demand scaled by the damping and density-discontinuity compensation, i.e.
\begin{align*}
    f_{in}^{(i)}(t) = \sum_{v_{r} \in \tilde{c}(v_i)} \gamma_{i}^{r}(t) \mathbb{E}\left[D_{\tilde{t}(v_i,v_{r},t)}^{(v_{r})} | \mathcal{F}_{\hat{t}_j} \right],
\end{align*}
we end up with the distribution parameters from MS2.

\section{Numerical study}
\label{cha.num}
In this section we present simulation results for the theoretical investigations presented in Sections \ref{cha4} and \ref{cha.3} on the optimal injection and the different information settings. We compare a simulation using the optimal injections and distribution parameters derived in Section \ref{cha.3} with a scenario in which we solve the deterministically reformulated optimization problems (\ref{ControlProblem111}) using the Matlab routine \textit{fmincon}. In the end, we focus on the comparison between demands given by the Jacobi process and the Ornstein-Uhlenbeck process.

We start with the introduction of the numerical discretization for the SDE given by the constraint (\ref{Demand}).
\paragraph{Discretization of (\ref{Demand})} The Jacobi process in the form of (\ref{Jacobi2}) which represents the constraint $(\ref{Demand})$ can be approximated using an Euler-Maruyama-scheme on a time grid $(t_j)_{j \in \mathbb{N}}$ with $t_j + t_0 + j\Delta t$ for some $\Delta t>0$ small enough:
\begin{align*}
    Z_{j+1} = Z_j + \Delta t \kappa (\theta(t_j) - Z_j) + \sigma \sqrt{\Delta t Z_j (1-Z_j)} X_j,
\end{align*}
where $X_j$ is a realization of a standard normal distributed random variable. Since the standard normal distribution takes values in an unbounded interval, we need to avoid values outside the interval $[0,1]$ for $Z_{j+1}$. Therefore, we add a truncation into the Euler-Maruyama scheme such that the process is reflected back into $[0,1]$
\begin{align*}
    Z_{j+1} &= \begin{cases} 1, &Z_{j+1}^* \geq 1\\
    Z_{j+1}^*, & Z_{j+1}^* \in (0,1) \\
    0, & Z_{j+1}^* \leq 0,\end{cases}
\end{align*}
where $Z_{j+1}^* = Z_j + \Delta t \kappa (\theta(t_j) - Z_j) + \sigma \sqrt{\Delta t Z_j (1-Z_j)} X_j$. 

\paragraph{Discretization of (\ref{NetworkDynamics})}
For the numerical discretization of the network dynamics from (\ref{NetworkDynamics}) we use a splitting algorithm to separate the flow dynamics and the damping effects.  
To solve the flow dynamics on arc $l$ we use an adaptive upwind scheme on a time grid $(t_j^l)_{j\in \mathbb{N}}$ with $t_{j+1}^l = t_j^l + \Delta t_j^l$ and the spatial discretization $(x_i)_{i\in \mathbb{N}}$ where $x_i = x_0 + i \Delta x$. Note that the time grid of the stochastic differential equation and the partial differential equation do not have to coincide. The step sizes are chosen such that the CFL condition is satisfied with equality in every time step, i.e. $\tfrac{\Delta t_j^{l}}{\Delta x} \lambda_l(t_j) = 1$. Therefore, the temporal grids depend on the velocity functions of the particular arc:
\begin{align*}
    \tilde{z}_i^{(l),j+1} = z_i^{(l),j} + \frac{\Delta t^{l}_j}{\Delta x} \lambda_l(t_j^l) \left(z_i^{(l),j} - z_{i-1}^{(l),j} \right).
\end{align*}
In a second step we take into account the damping and calculate $z_i^{(l),j+1}$ by
\begin{align*}
    z_i^{(l),j+1} = (1-\Delta t^{l}_j \mu_l(t_j)) \tilde{z}_i^{(l),j+1}.
\end{align*}
The choice of the upwind scheme is reasonable here, since we have linear dynamics with a fixed direction of movement. However, because of the different temporal step sizes on the different arcs there might be a mismatch in the time grids of two consecutive arcs. To ensure flux conservation at every node it is necessary to align the temporal grids of the discretization. Therefore, we introduce a time grid $(t_j^{\#})_{j \in \mathbb{N}}$ with fixed step size $\Delta t^{\#}>0$ and $t_j^{\#} = t_0 + j\cdot \Delta t^{\#}$ which is not dependent on the velocity functions of the arcs. At every node $v_l$ the ingoing flux $f_{in}^{(v_l)}$ on this time grid is calculated by
\begin{align*}
    f_{in}^{(v_l)}(t_j^{\#}) = z_{\text{end}}^{(l),k}\lambda_l(t_k^{l}) \frac{t_{k+1}^{l} - t_j^{\#}}{\Delta t^l_j} + z_{\text{end}}^{(l),k+1}\lambda_l(t_{k+1}^{(l)}) \frac{ t_j^{\#} - t_k^{(l)}}{\Delta t^l_j},~~\text{for }t_j^{\#} \in [t_k^{l}, t_{k+1}^{l}),
\end{align*}
where $z_{\text{end}}^{(l),k}$ denotes the density at the last spatial grid point of arc $l$ at $t_k^l$. 
Analogously, the outgoing fluxes of node $v_l$ given by $ f_{out}^{(v_l)}$  have to fulfill
\begin{align*}
    f_{out}^{(v_l)}(t_j^{\#}) = \sum_{m \in J^{out}_{v_l}} z_{1}^{(m),k}\lambda_m(t_k^{(m)}) \frac{t_{k+1}^{(m)} - t_j^{\#}}{\Delta t^{\#}} + z_{1}^{(m),k+1}\lambda_m(t_{k+1}^{(m)}) \frac{ t_j^{\#} - t_k^{(m)}}{\Delta t^{\#}},~\text{for }t_j^{\#} \in [t_k^{m}, t_{k+1}^{m}).
\end{align*}
The values of $z_1^{(m),k}$ act as the boundary conditions for arc $m \in J^{out}_{v_l}$ and are calculated using the distribution parameter $\alpha_{l,m}(t_j^m)$ and a weighted temporal average by
\begin{align*}
    z_1^{(m),k}(t_j^m) = \alpha_{l,m}(t_j^m)\left(\frac{f_{out}^{(v_l)}(t_k^{\#})}{\lambda_m(t_k^{\#})} \frac{t_j^m - t_k^{\#}}{\Delta t^{\#}} + \frac{f_{out}^{(v_l)}(t_{k+1}^{\#})}{\lambda_m(t_{k+1}^{\#})} \frac{t_{k+1}^{\#} - t_j^m}{\Delta t^{\#}} \right) ,~~\text{for }t_j^{m} \in [t_k^{\#}, t_{k+1}^{\#}).
\end{align*}

The reformulated and deterministic objective function of the optimization problem in (\ref{ofunc}) and the optimization problems resulting from MS3 in (\ref{optContMS3new}), which can be similarly interpreted as a deterministic optimization problem, are calculated using the \textit{fmincon} solver from Matlab R2021a.

To validate whether our numerical study works in expectation, we perform a Monte Carlo simulation with $N=500$ samples and consider a time period from $t_0=0$ to $T=2.5$. For the Euler-Maruyama scheme a temporal step size of $\Delta t = 10^{-4}$ is chosen. In the Upwind-scheme for the discretization of the dynamics on the arcs we choose $\Delta x = 5 \cdot 10^{-3}$ and adjust the temporal step size such that the CFL-condition holds exactly. 

As an error measure we consider a normalized root mean squared error of the deviation between the actual demand and the supply for any demand node $v_i$. The integration is performed on the interval starting at the first time an injected unit reaches the demand node up to the terminal time $T$ and approximated using a rectangular rule. For better comparability we normalize the error such that the time an injected unit reaches the demand node does not influence the error measure significantly., i.e
\begin{align*}
    \text{normRMSE}_i(z^{(i)}(1,\cdot)) = \frac{1}{T - \tilde{t}(v_0,v_i,t_0)}\int_{\tilde{t}(v_0,v_i,t_0)}^T \sqrt{\mathbb{E}\left[\left(D_s^{(v_i)} - f^{(i)}(z^{(i)}(1,s),s) \right)^2 | \mathcal{F}_{\hat{t}_j}\right]  }ds .
\end{align*}

\subsection{Simulation results for a 1-2 network}
We start with the presentation of results for a 1-2-network shown in Figure $\ref{network1-2}$. Even though this seems to be a very small example, the main characteristics are observable.

\paragraph{Deterministic demand}
For validity purposes, we define a benchmark framework in which we do not consider any stochasticity in the demand, i.e. $\sigma^{(v_i)} = 0$. Furthermore, we assume the velocities on the network arcs to be independent of time and constant all over the network. We investigate a setting without any damping term ($\mu_i^1(t)=0$) and compare it to a setting with constant and time-independent damping ($\mu_i^2(t)=0.4$) in the supply network. The choices of the parameters can be found in Table \ref{params1}.

\begin{table}[h!]
    \centering
    \begin{tabular}{c|c|c|c||c|c|c|c|c}
    \hline
         arc & $\lambda_i(t)$ & $\mu_i^1(t)$& $\mu_i^2(t)$ & node &$\theta^{(v_i)}(t)$ & $\kappa^{(v_i)}$ & $\sigma^{(v_i)}$ & $(d_0)^{(v_i)}$    \\
         \hline
         \hline
         1 & $14$ & $0$&$0.4$&-&-&-&-&- \\
         \hline
         \hline
         2 & $14$ &$0$&$0.4$ &$v_2$& $0.45 + 0.2\sin(t \pi+1)$ & $0$& $0$ & $0.4$\\
         \hline
         3 & $14$ &$0$&$0.4$ &$v_3$& $0.5 + 0.3\sin(t\pi-0.5)$ &$0$ & $0$ & $0.6$\\
         \hline
    \end{tabular}
    \caption{Parameter choices for the 1-2 network in the first example.}
    \label{params1}
\end{table}

Figure \ref{figSett1} shows the input control at the source node $v_0$ for the case with and without damping when we solve the optimal control problem (\ref{ControlProblem111}) using \textit{fmincon}. A simulation with the optimal inflow and distribution parameters from Section \ref{cha.3} leads to almost exactly the same result. We also present the comparison of supply in all the model settings (MS1, MS2, MS3) and demand at the two demand nodes $v_2$ and $v_3$. Since the demand is purely deterministic in this setup supply matches demand almost perfectly in all three cases. Due to the deterministic setting, the demand evolution is known in advance and no further improvement can be achieved when adding demand updates. We do not observe any difference in supply comparing the damped system with the system in which there is no damping, therefore we present only one figure for supply and demand at each node. This is an expected effect since the loss that occurs due to the damping is known in advance and taken into account when we determine the control. We will observe this pattern not only in this simple deterministic setting but throughout the whole section. The effect of the damping can be recognized in the two different controls. In the damped scenario the control has to be larger than in the undamped case. Note that the controls of MS2 and MS3 coincide.

\begin{figure}[h!]
    \centering
    \includegraphics[width=\textwidth]{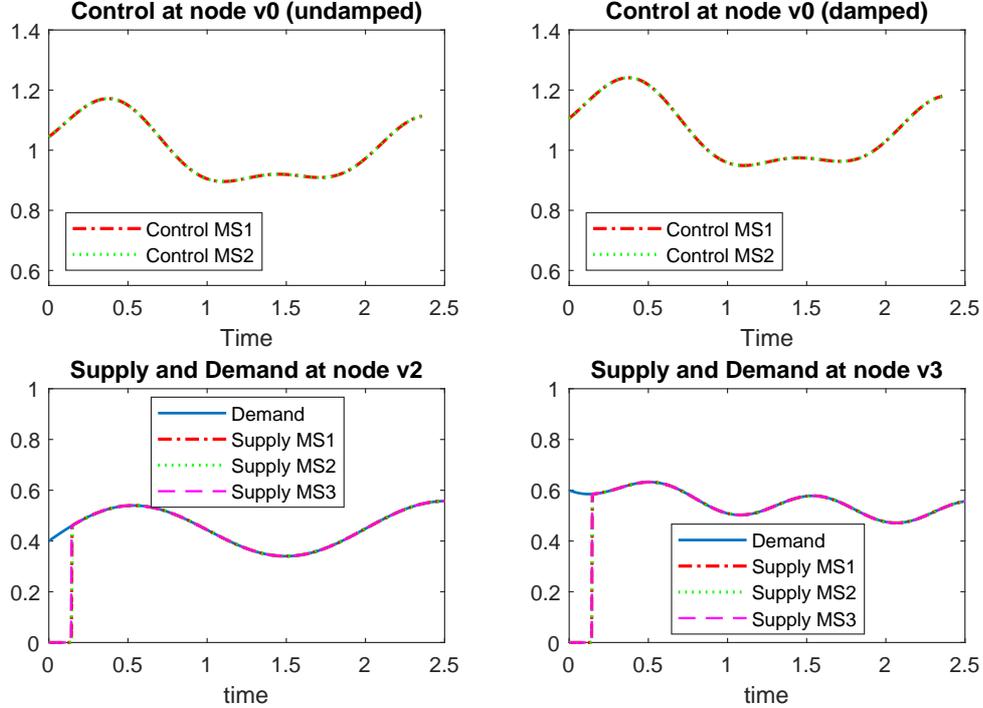}
    \caption{Control at the source node $v_0$ for the undamped scenario (left) and the damped scenario (right) and the supply in the three different settings plotted together with the actual demand from the first example.}
    \label{figSett1}
\end{figure}

We compare the normalized root mean squared errors in a Monte Carlo simulation with $N=500$ runs of the different information and update scenarios using the optimal conditions and \textit{fmincon}. For MS2 and MS3 we choose update intervals of $\Delta t_{up} = \frac{5}{14}$ which correspond to 6 updates. The normalized root mean squared errors with superscript 1 and superscript 2 correspond to the undamped and damped scenario, respectively. Since for a deterministic demand process all three levels of information coincide the error measures look the same for all cases, i.e. for the \textit{fmincon} study normRMSE$_2^1=$ normRMSE$_2^2 = 0.795\cdot 10^{-4}$ and normRMSE$_3^1=$ normRMSE$_3^2 = 0.155\cdot 10^{-4}$. The deviations from these error measures for using the calculated optimal values are negligible.

\paragraph{Stochastic demand and non-constant velocities and damping} As a second example we discuss a scenario, where on the one hand we have highly fluctuating demand and on the other hand the dynamics which differ from arc to arc may be dependent on the time. Table \ref{params2} shows the particular choices of the parameters. Similar to the first example we distinguish between a scenario with and without damping. For MS2 and MS3 we choose update intervals of $\Delta t_{up} = \frac{5}{14}$ which correspond to 6 updates.

\begin{table}[h!]
    \centering
    \scalebox{0.9}{
    \begin{tabular}{c|c|c|c||c|c|c|c|c}
    \hline
         arc & $\lambda_i(t)$ & $\mu_i^1(t)$& $\mu_i^2(t)$ & node &$\theta^{(v_i)}(t)$ & $\kappa^{(v_i)}$ & $\sigma^{(v_i)}$ & $(d_0)^{(v_i)}$    \\
         \hline
         \hline
         1 & $14 + \sin(2\pi t)$ & $0$&$0.4 + 0.2\sin(\pi t)$&-&-&-&-&- \\
         \hline
         \hline
         2 & $12 + \sin(2\pi t)$ &$0$&$0.5+ 0.2\sin(\pi t)$ &$v_2$& $0.45 + 0.2\sin(t \pi+1)$ & $2$& $\frac{9}{4}$ & $0.4$\\
         \hline
         3 & $12 + \sin(4\pi t)$ &$0$&$0.5+ 0.3\sin(\pi t)$ &$v_3$& $0.5 + 0.3\sin(t\pi-0.5)$ &$1$ & $\frac{3}{2}$ & $0.6$\\
         \hline 
    \end{tabular}}
    \caption{Parameter choices for the 1-2 network in the second example.}
    \label{params2}
\end{table}

Figure \ref{figSett2} shows one particular realization of the demand process with the corresponding controls and supplies at the demand nodes. Again, we provide the illustrations for the study using \textit{fmincon}. The results for the scenario with optimal parameters does not differ significantly. Comparing the controls we observe that they are generally larger in the setting with damping. Additionally, we observe jumps in the controls of MS2 which correspond to new demand information. The demand at node $v_2$ generally evolves lower than in mean and approaches 0 occaisonally. Therefore, the supply given by MS2 and MS3 fit demand much better than supply from the unupdated setting MS1. 
\begin{figure}[h!]
    \centering
    \includegraphics[width=\textwidth]{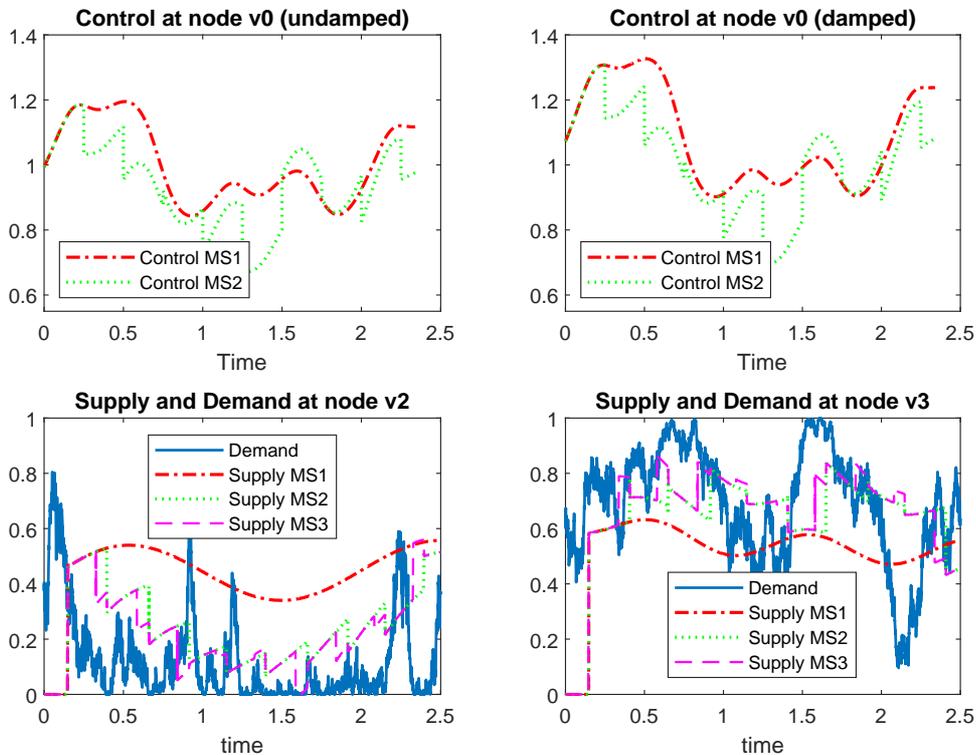}
    \caption{Control at the source node $v_0$ for the undamped scenario (left) and the damped scenario (right) and the supply in the three different settings plotted together with the actual demand from the second example.}
    \label{figSett2}
\end{figure}

To validate whether the update strategies provide smaller errors we again perform a Monte Carlo simulation of $N=500$ runs and compare the normalized RMSE using \textit{fmincon} and the optimal values in Table \ref{ErrTab28} and \ref{ErrTab29}. 
\begin{table}[htb!]
\centering
\begin{tabular}{l||c|c||c|c}
\hline
& normRMSE$_2^1$ & normRMSE$_3^1$ & normRMSE$_2^2$ & normRMSE$_3^2$ \\
\hline
\hline
MS1 & $0.3635$ & $0.3561$ &  $0.3635$ & $0.3561$\\
\hline
\hline
MS2 & $0.3175$ & $0.2612$ &  $0.3175$ & $0.2612$\\
\hline
\hline
MS3 & $0.3016$ & $0.2525$ & $0.3016$ & $0.2525$\\
\hline
\end{tabular}
\caption{normRMSE for the different Model Settings (MS) in the 1-2-network in example 2 using fmincon for the optimization problems.}
\label{ErrTab28}
\end{table}

\begin{table}[htb!]
\centering
\begin{tabular}{l||c|c||c|c}
\hline
& normRMSE$_2^1$ & normRMSE$_3^1$ & normRMSE$_2^2$ & normRMSE$_3^2$ \\
\hline
\hline
MS1 & $0.3579$ & $0.3488$ &  $0.3579$ & $0.3488$\\
\hline
\hline
MS2 & $0.3155$ & $0.2614$ &  $0.3155$ & $0.2614$\\
\hline
\hline
MS3 & $0.2928$ & $0.2517$ & $0.2928$ & $0.2517$\\
\hline
\end{tabular}
\caption{normRMSE for the different Model Settings (MS) in the 1-2-network in example 2 using the calculated optimal values for the inflow and the distribution parameters.}
\label{ErrTab29}
\end{table}
The error of the damped and undamped scenarios coincide and we observe, as expected, that the error reduces when more information is taken into account. 

We remark that the run-time for the implementation is mainly driven by the time that is required to solve the optimization problems using fmincon. For example, the computation for a sample of $N=100$ runs for the 1-2 network in the case of MS1 and MS2 takes about 220 minutes on a standard desktop PC with a CPU of 3.19 Ghz using the above chosen parameters. The calculation for MS3 scales approximately by factor 2 due to the additional optimization problem that has to be solved at node $v_1$. A simulation using the optimal parameters takes depending on the model setting between 20 and 40 seconds.

In Table \ref{errTab5} we show the evolution of the error reduction for varying the number of updates. Since this has no influence on MS1 we only consider MS2 and MS3. The tables show the reduction of the errors compared to a setting without updates. It can be observed that the errors are significantly reduced using more and more updates. But since the increases in the reduction get smaller when doubling the number of updates at a high level, it seems reasonable to assume that there is some base error that cannot be undercut.

\begin{table}[htb!]
\centering
\begin{tabular}{c|c}
\begin{tabular}{r||c|c}
\hline
$v_2$& MS2 & MS3 \\
\hline
\hline
no update  & $0.3579$ & $0.3579$\\
\hline
1 update & $2.71\%$ & $3.16\%$\\
\hline
2 updates & $4.86\%$ & $5.81\%$\\
\hline
3 updates  & $6.68\%$ & $8.19\%$\\
\hline
6 updates & $11.85\%$ & $14.47\%$\\
\hline
12 updates  & $17.83\%$ & $22.24\%$\\
\hline
24 updates & $22.10\%$ & $28.28\%$\\
\hline
48 updates  & $24.98\%$ & $32.33\%$\\
\hline
96 updates & $26.71\%$ & $34.81\%$\\
\hline
192 updates  & $27.71\%$ & $36.18\%$\\
\hline
384 updates  & $28.25\%$ & $36.94\%$\\
\hline
\end{tabular}
\begin{tabular}{r||c|c}
\hline
 $v_3$& MS2 & MS3 \\
\hline
\hline
no update  & $0.3488$ & $0.3488$\\
\hline
1 update  & $7.45\%$ & $8.20\%$\\
\hline
2 updates  & $12.79\%$ & $14.28\%$\\
\hline
3 updates  & $17.09\%$ & $18.98\%$\\
\hline
6 updates  & $25.06\%$ & $27.78\%$\\
\hline
12 updates  & $31.22\%$ & $34.86\%$\\
\hline
24 updates  & $36.87\%$ & $41.08\%$\\
\hline
48 updates  & $40.02\%$ & $44.47\%$\\
\hline
96 updates  & $41.89\%$ & $46.39\%$\\
\hline
192 updates  & $42.98\%$ & $47.50\%$\\
\hline
384 updates  & $43.46\%$ & $48.02\%$\\
\hline
\end{tabular}
\end{tabular}
\caption{Error reduction for normRMSE at $v_2$ and $v_3$ compared to the setting without updates.}
\label{errTab5}
\end{table}

\subsection{Comparison to an Ornstein-Uhlenbeck demand process}
Next, we perform a demand simulation using the Ornstein-Uhlenbeck process introduced in $(\ref{OU-Process})$ and compare it with the results we obtained from the Jacobi demand. The Ornstein-Uhlenbeck process can be also discretized by an Euler Maruyama-scheme on a time grid $(t_j)_{j \in \mathbb{N}}$ with $t_j + t_0 + j\Delta t$ for some $\Delta t>0$ small enough:
\begin{align*}
    \hat{Z}_{j+1} = \hat{Z}_j + \Delta t \hat{\kappa} (\hat{\theta}(t_j) - \hat{Z}_j) + \hat{\sigma} X_j,
\end{align*}
where $X_j$ is a realization of standard normal distributed random variable.

To compare the Jacobi demand with a demand generated from the Ornstein-Uhlenbeck process, we assume the same parameters for $\kappa, \theta$ and initial demand as in Table \ref{params2}. The stochastic perturbation size in the Jacobi process depends on $\sigma$ as well as the state of the process itself and can not be transferred analogously since the stochastic perturbation in the Ornstein-Uhlenbeck process is exclusively given by a factor $\hat{\sigma}$. We use $\hat{\sigma}^{(v_2)} = 0.14, \hat{\sigma}^{(v_3)} = 0.1$ for the Ornstein-Uhlenbeck process which represents a comparable choice for the intensity of the stochastic fluctuations.

\begin{figure}[h]
    \centering
    \includegraphics[width=\textwidth]{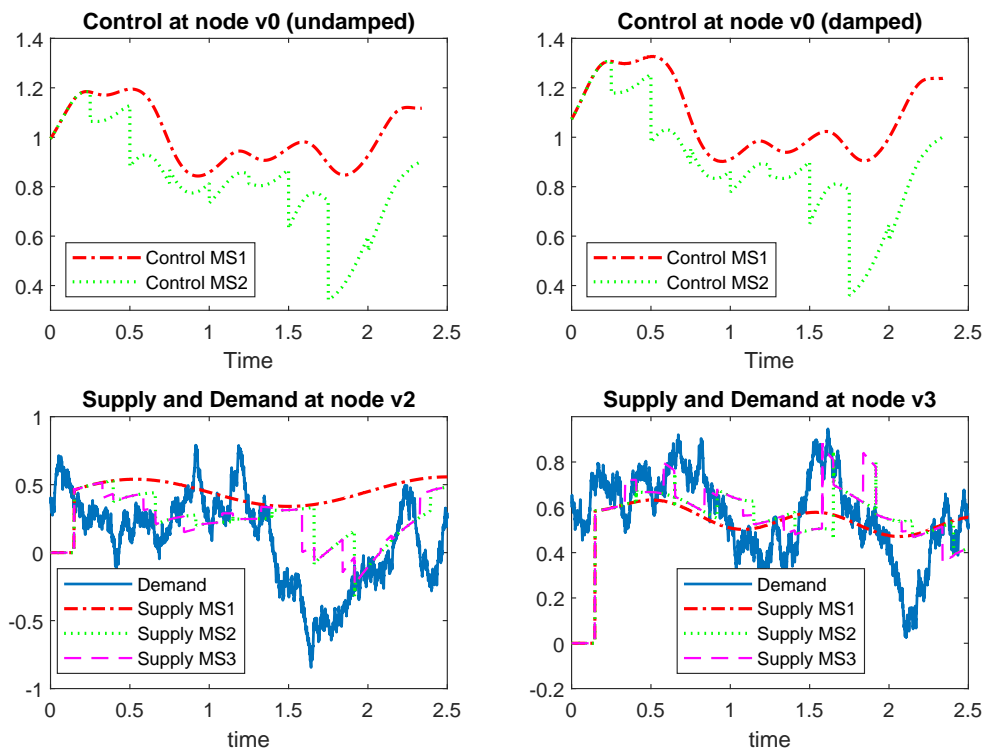}
    \caption{Control at the source node $v_0$ for the undamped scenario (left) and the damped scenario (right) and the supply in the three different settings plotted together with the actual demand from the third example using an Ornstein-Uhlenbeck process.}
    \label{figSett2OU}
\end{figure}

For the same realizations in the probability space as for the simulation that corresponds to Figure \ref{figSett2}, Figure \ref{figSett2OU} shows the Ornstein-Uhlenbeck demand process with the corresponding controls and supplies. In contrast to before, at node $v_2$ negative demand occurs around $t=1.5$, and even negative supplies around $t=2$, due to the unboundedness of the Ornstein-Uhlenbeck process. This effect does not have a natural interpretation and shows the drawbacks that may arise using the Ornstein-Uhlenbeck process as a demand process. Since we work with the same realizations in the probability space as in the previous example for the Jacobi process, we can compare both demands and observe that the Jacobi process also shows low demands in this area, but due to its characteristics always stays non-negative.

\section{Conclusion}
In this work, we have derived an explicit optimal control strategy, depending on the level of information, for supply networks with uncertain demand. The properties of the Jacobi process, in contrast to a demand governed by an Ornstein-Uhlenbeck process, are used to guarantee a reasonable interpretation of demand and supply in the case of supply networks. An explicit representation of the optimal input and the distribution parameters allows for a suitable and efficient numerical treatment. Numerical examples illustrate the main characteristics of the optimal control problem.

Future work includes the consideration of transport networks with nonlinear dynamics and uncertain demand using the Lax-Hopf technique \cite{Claudel2020b}.

\section*{Declarations}

\paragraph{Funding}
This work was supported by the DAAD project "Stochastic dynamics for complex networks and systems" (Project-ID 5744394).

\paragraph{Conflict of interest}
The authors declare no competing interests.

\paragraph{Author contribution}
All authors have contributed equally to this article.


\section{Appendix}
We present the detailed calculation of the second moment for a Jacobi process with time-varying mean reversion level given by the SDE
\begin{align}
\label{SDE_theta}
    dZ_t = \kappa (\theta(t) - Z_t)dt + \sigma \sqrt{Z_t(1-Z_t)}dW_t.
\end{align}
    
    \begin{lemma}
     Let $(\bar{\theta}_n)_{n \in \mathbb{N}}$ be a sequence of step functions converging uniformly to a function $\theta \in C^1({[t_0,T]})$. Additionally, let $\underset{n \rightarrow \infty}{\lim} t_n = T$. Then, the conditional second moment for the solution of $(\ref{SDE_theta})$ is given by
    \begin{align*}
    \mathbb{E}[Z_T^2 | Z_{t_0}=z_0] 
    &=  \int_{t_0}^T (2\kappa \theta(s) + \sigma^2)\left(z_0e^{-\kappa(s-t_0)} + \kappa \int_{t_0}^s  \theta(r)e^{-\kappa(s-r)}dr \right)e^{-(2\kappa + \sigma^2)(T-s)}ds \\
    &~~~~+ z_0^2e^{-(2\kappa + \sigma^2)(T-t_0)}.
        \end{align*}
    \end{lemma}
    
    \begin{proof}
    For a constant mean reversion level the conditional second moment is presented in equation (\ref{condSecMoment}). The idea of the proof is to use this expression to find a representation of the second moment for piecewise constant mean reversion levels and then use a uniform limit to show the result for continuously differentiable functions $\theta$. First, assume that for $t_0 < t_1< \ldots < t_n$ on a bounded interval $[t_0,t_n]$, $\theta$ is a step-function, i.e.
    \begin{align}
    \label{thetastep}
        \theta(t) = \sum_{i=0}^{n-1} \theta_i \mathbbm{1}_{[t_i,t_{i+1})}(t)
    \end{align}
    for $\theta_i \in \mathbb{R}$.
    For $n=1$, we obtain the the conditional second moment as in $(\ref{condSecMoment})$. We use an induction to calculate the conditional second moment for an arbitrary $n+1 \in \mathbb{N}$ assuming that the conditional second moment is known for a step-function with $n$ steps is given by
    \begin{align}
    \begin{split}
    \label{scndmm}
    &\mathbb{E}[Z_{t_n}^2 | Z_{t_0} = z_0] =  \sum_{i=0}^{n-1}\Bigg[\frac{(2 \kappa \theta_i + \sigma^2)\theta_i}{2 \kappa + \sigma^2} e^{-(2\kappa + \sigma^2)(t_n - t_{i+1})} + \frac{\kappa \theta_i (2 \kappa \theta_i + \sigma^2)}{(2\kappa + \sigma^2)(\kappa + \sigma^2)} e^{-(2\kappa + \sigma^2)(t_n - t_{i})}  \Bigg]\\
    &~~~~ + \sum_{i=1}^{n-1} \Bigg[ \frac{2\kappa \theta_i + \sigma^2}{\kappa + \sigma^2} \Bigg(\sum_{j=1}^{i} (\theta_{j-1} - \theta_j)e^{-\kappa(t_{i+1} - t_j)-(2\kappa + \sigma^2)(t_n - t_{i+1})}  -\theta_{i-1}e^{-(2\kappa + \sigma^2)(t_n - t_i)} \\
    &~~~~ - \sum_{j=1}^{i-1}(\theta_{j-1} - \theta_j)e^{-\kappa(t_i - t_j) - (2\kappa + \sigma^2)(t_n - t_{i})}  \Bigg)\Bigg] \\
    & ~~~~ + \sum_{i=1}^{n-1} \Bigg[\frac{2\kappa \theta_i + \sigma^2}{\kappa + \sigma^2}(z_0-\theta_0)\left(e^{-\kappa(t_{i+1} - t_0)-(2\kappa + \sigma^2)(t_n - t_{i+1})}  -e^{-\kappa(t_i - t_0)-(2\kappa + \sigma^2)(t_n - t_{i})}\right)\Bigg] \\
    &~~~~+\frac{2\kappa \theta_0 + \sigma^2}{\kappa + \sigma^2} \left((z_0 - \theta_0)e^{-\kappa(t_1 - t_0) - (2\kappa + \sigma^2)(t_n-t_1)} -z_0 e^{-(2\kappa + \sigma^2)(t_n-t_0)}\right) + z_0^2e^{-(2\kappa + \sigma^2)(t_n - t_0)}.
    \end{split}
    \end{align}
    
    Then the induction step to $n+1$ reads
    \begin{align*}
    &\mathbb{E}\left[Z_{t_{n+1}}^2 | Z_{t_0}= z_0 \right] = \mathbb{E}\left[Z_{t_{n+1}}^2 - Z_{t_{n}}^2 + Z_{t_{n}}^2 | Z_{t_0}= z_0 \right]
    = \mathbb{E}\left[\mathbb{E}\left[Z_{t_{n+1}}^2  | Z_{t_{n}} = Z_{n}\right] ~| Z_{t_0}= z_0 \right] \\
    &= \mathbb{E} \Bigg[\frac{(2\kappa \theta_{n} + \sigma^2)\theta_{n}}{2\kappa + \sigma^2} + \frac{2\kappa \theta_{n} + \sigma^2}{\kappa + \sigma^2}(Z_{n} - \theta_{n}) e^{-\kappa(t_{n+1}-t_{n})} \\
    &~~~~+ \left(Z_{n}^2 - \frac{2\kappa\theta_{n} + \sigma^2}{\kappa + \sigma^2}Z_{n} + \frac{\kappa \theta_{n} (2\kappa \theta_{n} + \sigma^2)}{(2\kappa + \sigma^2)(\kappa + \sigma^2)}  \right)e^{-(2\kappa + \sigma^2)(t_{n+1}-t_{n})}  \Big| Z_{t_0}=z_0  \Bigg]  \\
    &= \frac{(2\kappa \theta_n + \sigma^2)\theta_n}{2\kappa + \sigma^2}+ \frac{2\kappa \theta_n + \sigma^2}{\kappa + \sigma^2} \bigg(\sum_{j=1}^n (\theta_{j-1} - \theta_j)e^{-\kappa(t_{n+1} - t_j)}  \bigg) \\
    &~~~~- \frac{2\kappa \theta_n + \sigma^2}{\kappa + \sigma^2}\bigg(\theta_{n-1}e^{-(2\kappa + \sigma^2)(t_{n+1} - t_n)}+ \sum_{j=1}^{n-1} (\theta_{j-1} - \theta_j)e^{-\kappa(t_{n}-t_j) - (2\kappa + \sigma^2)(t_{n+1} - t_n)} \bigg) \\
    &~~~~+ \frac{\kappa\theta_n(2\kappa \theta_n + \sigma^2)}{(2\kappa + \sigma)(\kappa + \sigma^2)} e^{-(2\kappa + \sigma^2)(t_{n+1} - t_n)} + \frac{2\kappa \theta_n + \sigma^2}{\kappa + \sigma^2}\left(z_0-\theta_0\right)\left(e^{-\kappa(t_{n+1} - t_0)} - e^{-\kappa(t_{n} - t_0)}\right) \\
    &~~~~ + \Bigg[\sum_{i=0}^{n-1}\Bigg[\frac{(2 \kappa \theta_i + \sigma^2)\theta_i}{2 \kappa + \sigma^2} e^{-(2\kappa + \sigma^2)(t_n - t_{i+1})} + \frac{\kappa \theta_i (2 \kappa \theta_i + \sigma^2)}{(2\kappa + \sigma^2)(\kappa + \sigma^2)} e^{-(2\kappa + \sigma^2)(t_n - t_{i})}  \Bigg]\\ 
    &~~~~ + \sum_{i=1}^{n-1} \Bigg[ \frac{2\kappa \theta_i + \sigma^2}{\kappa + \sigma^2} \Big(\sum_{j=1}^{i} (\theta_{j-1} - \theta_j)e^{-\kappa(t_{i+1} - t_j)-(2\kappa + \sigma^2)(t_n - t_{i+1})} \Big) \\
    &~~~~- \frac{2\kappa \theta_i + \sigma^2}{\kappa + \sigma^2} \Big(\theta_{i-1}e^{-(2\kappa + \sigma^2)(t_n - t_i)}+ \sum_{j=1}^{i-1}(\theta_{j-1} - \theta_j)e^{-\kappa(t_i - t_j) - (2\kappa + \sigma^2)(t_n - t_{i})} \Big) \Bigg] \\
    & ~~~~ + \sum_{i=1}^{n-1} \Bigg[\frac{2\kappa \theta_i + \sigma^2}{\kappa + \sigma^2} \Bigg((z_0-\theta_0)\left(e^{-\kappa(t_{i+1} - t_0)-(2\kappa + \sigma^2)(t_n - t_{i+1})} -e^{-\kappa(t_i - t_0)-(2\kappa + \sigma^2)(t_n - t_{i})}\right) \Bigg) \Bigg]\\
    &~~~~ +\frac{2\kappa \theta_0 + \sigma^2}{\kappa + \sigma^2} \left((z_0 - \theta_0)e^{-\kappa(t_1 - t_0) - (2\kappa + \sigma^2)(t_n-t_1)} -z_0 e^{-(2\kappa + \sigma^2)(t_n-t_0)}\right) \\
    &~~~~+ z_0^2e^{-(2\kappa + \sigma^2)(t_n - t_0)} \Bigg]e^{-(2\kappa + \sigma^2)(t_{n+1} - t_{n})}.
    \end{align*}
    Summarizing, we obtain the proposed equation ($\ref{scndmm}$). As a last step we calculate the limit for $n \rightarrow \infty$.
    \begin{align*}
    &\mathbb{E}[Z_{T}^2 | Z_{t_0} = z_0] \\&= \underset{n \rightarrow \infty}{\lim} ~ \sum_{i=0}^{n-1}\Bigg[\frac{(2 \kappa \theta_i + \sigma^2)\theta_i}{2 \kappa + \sigma^2} e^{-(2\kappa + \sigma^2)(t_n - t_{i+1})} + \frac{\kappa \theta_i (2 \kappa \theta_i + \sigma^2)}{(2\kappa + \sigma^2)(\kappa + \sigma^2)} e^{-(2\kappa + \sigma^2)(t_n - t_{i})}  \Bigg]\\
    &~~~~ + \sum_{i=1}^{n-1} \Bigg[ \frac{2\kappa \theta_i + \sigma^2}{\kappa + \sigma^2} \Bigg(\sum_{j=1}^{i} (\theta_{j-1} - \theta_j)e^{-\kappa(t_{i+1} - t_j)-(2\kappa + \sigma^2)(t_n - t_{i+1})}  -\theta_{i-1}e^{-(2\kappa + \sigma^2)(t_n - t_i)} \\
    &~~~~ - \sum_{j=1}^{i-1}(\theta_{j-1} - \theta_j)e^{-\kappa(t_i - t_j) - (2\kappa + \sigma^2)(t_n - t_{i})}  \Bigg)\Bigg] \\
    & ~~~~ + \sum_{i=1}^{n-1} \Bigg[\frac{2\kappa \theta_i + \sigma^2}{\kappa + \sigma^2}(z_0-\theta_0)\left(e^{-\kappa(t_{i+1} - t_0)-(2\kappa + \sigma^2)(t_n - t_{i+1})}  -e^{-\kappa(t_i - t_0)-(2\kappa + \sigma^2)(t_n - t_{i})}\right)\Bigg] \\
    &~~~~+\frac{2\kappa \theta_0 + \sigma^2}{\kappa + \sigma^2} \left((z_0 - \theta_0)e^{-\kappa(t_1 - t_0) - (2\kappa + \sigma^2)(t_n-t_1)} -z_0 e^{-(2\kappa + \sigma^2)(t_n-t_0)}\right) \\ 
    &~~~~+ z_0^2e^{-(2\kappa + \sigma^2)(t_n - t_0)}\\  
    &= \frac{(2 \kappa \theta(t_0) + \sigma^2)\theta(t_0)}{2 \kappa + \sigma^2} \left(1+ \frac{\kappa}{\kappa + \sigma^2} \right) e^{-(2\kappa + \sigma^2)(T - t_{0})}  - \frac{2\kappa \theta(t_0) + \sigma^2}{\kappa + \sigma^2} \theta(t_0)e^{-(2\kappa + \sigma^2)(T - t_{0})} \\
    &~~~~+ \underset{n \rightarrow \infty}{\lim} \Bigg[ \sum_{i=1}^{n-1}\Bigg[\frac{2 \kappa \theta_i + \sigma^2}{(2 \kappa + \sigma^2)(\kappa + \sigma^2)} \bigg(  (\kappa + \sigma^2) \Big(\theta_i e^{-(2 \kappa + \sigma^2)(t_n - t_{i+1})} - \theta_{i-1} e^{-(2 \kappa + \sigma^2)(t_n - t_{i})}\Big) \\
    &~~~~ + \kappa \left( \theta_i - \theta_{i-1}  \right)e^{-(2\kappa + \sigma^2)(t_n - t_i)} \bigg)  \Bigg]+ \sum_{i=1}^{n-1}\frac{2\kappa \theta_i + \sigma^2}{\kappa + \sigma^2} (\theta_{i-1} - \theta_i)e^{ - (2\kappa + \sigma^2)(t_n - t_{i+1})} \\
    &~~~~ +  \sum_{i=1}^{n-1} \frac{2\kappa \theta_i + \sigma^2}{\kappa + \sigma^2} \sum_{j=1}^{i-1}(\theta_{j-1} - \theta_j)\left(e^{-\kappa(t_{i+1} - t_j) - (2\kappa + \sigma^2)(t_n - t_{i+1})} - e^{-\kappa(t_{i} - t_j) - (2\kappa + \sigma^2)(t_n - t_{i})} \right)  \\
    &~~~~ + \sum_{i=1}^{n-1} \Bigg[\frac{2\kappa \theta_i + \sigma^2}{\kappa + \sigma^2}(z_0-\theta_0)\left( e^{-\kappa(t_{i+1} - t_0)-(2\kappa + \sigma^2)(t_n - t_{i+1})} - e^{-\kappa(t_{i} - t_0)-(2\kappa + \sigma^2)(t_n - t_{i})} \right)\Bigg]\Bigg] \\
    &~~~~+ z_0^2e^{-(2\kappa + \sigma^2)(T - t_0)}\\
    &= \frac{(2 \kappa \theta(t_0) + \sigma^2)\theta(t_0)}{\kappa + \sigma^2}  e^{-(2\kappa + \sigma^2)(T - t_{0})}- \frac{2\kappa \theta(t_0) + \sigma^2}{\kappa + \sigma^2}\theta(t_0)e^{-(2\kappa + \sigma^2)(T - t_{0})} \\
    &~~~~+ \int_{t_0}^T\frac{2 \kappa \theta(s) + \sigma^2}{(2 \kappa + \sigma^2)(\kappa + \sigma^2)}  \left( (2\kappa + \sigma^2) \left( \theta'(s) + (2\kappa + \sigma^2)\theta(s)\right) \right)e^{-(2\kappa + \sigma^2)(T-s)}ds \\
    &~~~~ + \int_{t_0}^T \frac{2\kappa \theta(s) + \sigma^2}{\kappa + \sigma^2} \left(-\theta'(s) + \int_{t_0}^s -\theta'(r)(\kappa + \sigma^2)e^{-\kappa(s-r)}dr \right)e^{ - (2\kappa + \sigma^2)(T - s)}ds\\
    &~~~~+ \int_{t_0}^T (2\kappa \theta(s) + \sigma^2)(z_0-\theta(t_0)) e^{-\kappa(s - t_0)-(2\kappa + \sigma^2)(T - s)} ds+ z_0^2e^{-(2\kappa + \sigma^2)(T - t_0)}\\
    &= \int_{t_0}^T (2\kappa \theta(s) + \sigma^2)\Big(\theta(s) - \int_{t_0}^s \theta'(r)e^{-\kappa(s-r)}dr  +(z_0 - \theta(t_0))e^{-\kappa(s-t_0)}  \Big)e^{-(2\kappa + \sigma^2)(T-s)} ds\\
    &~~~~+ z_0^2 e^{-(2\kappa + \sigma^2)(T - t_0)}\\
    &= \int_{t_0}^T (2\kappa \theta(s) + \sigma^2)\left(z_0e^{-\kappa(s-t_0)} + \kappa \int_{t_0}^s  \theta(r)e^{-\kappa(s-r)}dr\right) e^{-(2\kappa + \sigma^2)(T-s)}ds  \\
    &~~~~+z_0^2  e^{-(2\kappa + \sigma^2)(T-t_0)} \qedhere
    \end{align*}
    
    \end{proof}

\printbibliography
\end{document}